\documentclass[a4paper]{article}

\usepackage{arydshln}
\usepackage{amsmath}
\usepackage{amsfonts,amssymb}
\usepackage[compress]{cite}
\usepackage{graphicx}
\usepackage{mdwlist}
\usepackage{paralist}
\usepackage{url}

\newtheorem{lemma}{\sc Lemma}[section]
\newtheorem{theorem}[lemma]{\sc Theorem}

\newtheorem{remark}{\sc Remark}[section]
\newtheorem{assumption}{\sc Assumption}[section]
\newtheorem{definition}{\sc Definition}[section]

\newcommand{\jpfig}[4]{\begin{figure}[t] \centering \includegraphics[width=#1\linewidth]{#2} \caption{\label{#3}#4} \end{figure}}

\renewcommand{\Re}{\mathbb{R}}
\renewcommand{\matrix}[2]{\left[\begin{array}{#1} #2 \end{array}\right] }

\newcommand{\diag}{\text{diag}}

\newcommand{\comm}{\mathcal{C}}

\newcommand{\p}{\mathcal{P}}
\newcommand{\A}{\mathcal{A}}
\newcommand{\B}{\mathcal{B}}
\newcommand{\K}{\mathcal{K}}
\newcommand{\R}{\mathcal{R}}
\newcommand{\D}{\mathcal{D}}

\def\QED{~\rule[-1pt]{5pt}{5pt}\par\medskip}
\newenvironment{proof}{{\it Proof:\ }}{ \hfill \QED}

\DeclareMathOperator*{\argmin}{arg\,min}

\title{Decentralized Disturbance Accommodation with Limited Plant Model Information\thanks{The work of F.~Farokhi and K.~H.~Johansson were supported by grants from the Swedish Research Council and the Knut and Alice Wallenberg Foundation. The work of C.~Langbort was supported, in part, by the 2010 AFOSR MURI ``Multi-Layer and Multi-Resolution Networks of Interacting Agents in Adversarial Environments''.}}

\author{Farhad~Farokhi\thanks{ACCESS Linnaeus Center, School of Electrical Engineering, KTH Royal Institute of Technology, SE-100 44 Stockholm, Sweden. E-mails: \{farokhi,kallej\}@ee.kth.se }, C\'{e}dric~Langbort\thanks{Department of Aerospace Engineering, University of Illinois at Urbana-Champaign, Illinois, USA. E-mail: langbort@illinois.edu}, Karl~H.~Johansson$^\dag$ }

\begin{document}

\maketitle

\begin{abstract}
The design of optimal disturbance accommodation and servomechanism controllers with limited plant model information is considered in this paper. Their closed-loop performance are compared using a performance metric called competitive ratio which is the worst-case ratio of the cost of a given control design strategy to the cost of the optimal control design with full model information. It was recently shown that when it comes to designing optimal centralized or partially structured decentralized state-feedback controllers with limited model information, the best control design strategy in terms of competitive ratio is a static one. This is true even though the optimal structured decentralized state-feedback controller with full model information is dynamic. In this paper, we show that, in contrast, the best limited model information control design strategy for the disturbance accommodation problem gives a dynamic controller. We find an explicit minimizer of the competitive ratio and we show that it is undominated, that is, there is no other control design strategy that performs better for all possible plants while having the same worst-case ratio. This optimal controller can be separated into a static feedback law and a dynamic disturbance observer. For constant disturbances, it is shown that this structure corresponds to proportional-integral control.
\end{abstract}

\section{Introduction}
Recent advances in networked control systems have created new opportunities and challenges in controlling large-scale systems composed of several interacting subsystems. An example of a networked control system is shown in Figure~\ref{FigureNCS} where $P_i$ denotes the subsystems to be controlled and $C_i$ denotes the controllers. The interactions between the subsystems and the controllers as well as the external disturbances and references are indicated by arrows. For such networked systems, many researchers have considered the problem of decentralized or distributed stabilization or optimal control as well as the effect of communication channel limitations on closed-loop performance~\cite{Wang1973,Ozguler1990,Sandell1978,Mahajan2009123,Witsenhausen1968,Papadimitriou1986,Rotkowitz2006,Swigart2010,Shah2010,Voulgaris200351}. However, at the heart of all these methods lies the (sometimes implicit) assumption that the designer has access to the global plant model information when designing a local controller. This assumption might not be warranted, however, in some applications of interest~\cite{Dunbar2007,Negenborn2010}, in which the designer is constrained to compute local controllers for a large-scale systems in a distributed manner with access to only a limited or partial model of the plant. This might be due to several reasons, for example, \begin{inparaenum}[\upshape(\itshape i\upshape)] \item \label{case:3} the designer wants the parameters of each local controller to only depend on local model information, so that the controllers do not need to be modified if the model parameters of a particular subsystem, which is not directly connected to them, change, \item \label{case:2} the design of each local controller is done by a designer with no access to the global model of plant since at the time of design the complete plant model information is not available or might change later in the design process, or \item \label{case:1} different subsystems belong to different individuals who refuse to share their model information since they consider it private. \end{inparaenum} These situations are very common in practice. For instance, a chemical plant in process industry can have thousands of proportional-integral-derivative controllers. These processes well illustrate Case~\upshape(\itshape\ref{case:3}\upshape), as the tuning of each local controller does not typically require model information from other control loops in order to simplify the maintenance and limit the controller complexity. Case~\upshape(\itshape\ref{case:2}\upshape) is typical for cooperative driving such as vehicle platooning, where each vehicle has its own local (cruise) controller which cannot be designed based on model information of all possible vehicles that it may cooperate with in future traffic scenarios. Case~\upshape(\itshape\ref{case:1}\upshape) can be also illustrated by the control of the power grid, where economic incentives might limit the exchange of network model information across regional borders. Therefore, we have started investigating the concept of limited model information control design for large-scale systems~\cite{Langbort2010,Farokhi_ACC_2011,FarokhiLangbortJohansson2011,Farokhi2011}.

Control design strategies, mappings from the set of plants of interest to the set of applicable controllers, with various degrees of model information are compared using the competitive ratio as a performance metric, that is, the worst-case ratio of the cost of a given control design strategy to the cost of the optimal control design with full model information. In control design with limited plant model information, we search for the ``best'' control design strategy which attains the minimum competitive ratio among all limited model information design strategies. As this minimizer might not be unique, we further want to find an undominated minimizer of the competitive ratio, that is, there is no other control design strategy in the set of all limited model information design strategies with a better closed-loop performance for all possible plants while maintaining the same worst-case ratio. Recent attention has been on limited model information design methods that produce centralized or decentralized static state-feedback controllers with specific structure. This was justified, at first, by being the simplest case to explore~\cite{Langbort2010,Farokhi_ACC_2011,FarokhiLangbortJohansson2011}, and then, maybe more surprisingly, by the recently proven fact that the ``best'' (in the sense of competitive ratio and domination) state-feedback structured $\textsc{H}_2$- controller for a plant with lower triangular information pattern that can be designed with limited model information is also static~\cite{Farokhi2011}, even though the best such controller constructed with access to full model information is dynamic~\cite{Swigart2010,Shah2010}. In this paper, we study the problem of limited model information control design for optimal disturbance accommodation and servomechanism, and show that, contrary to the situations mentioned above, the ``best'' limited model information design method gives dynamic controllers. Optimal disturbance accommodation is a meaningful model for problems such as constant disturbance rejection or step reference tracking, and has been well-studied in the literature~\cite{Smith1972,Davison1971,Johnson1968,Young1972,Anderson1971}, but with no attention being paid to the model information limitations in the design procedure.

In this paper, specifically, we consider limited model information control design for interconnection of scalar discrete-time linear time-invariant subsystems being affected by scalar decoupled disturbances with a quadratic separable performance criterion. The choice of such a separable cost function is motivated first by the servomechanism and disturbance accommodation literature~\cite{Smith1972,Davison1971,Johnson1968,Young1972,Anderson1971}, and second by our interest in dynamically-coupled but cost-decoupled plants and their applications in supply chains and shared infrastructure~\cite{Dunbar2007,Negenborn2010} which has been shown to be well-modeled in this fashion. The assumptions on scalar subsystems and scalar disturbances are technical assumptions to make the algebra in the proofs shorter. Since we want each subsystem to be directly controllable (so that designing subcontrollers based on only local model information is possible), we assume that the overall system is fully-actuated.

We start with the case that each subcontroller is only designed with the corresponding subsystem model information. We prove that, in the case where the plant graph contains no sink and the control graph is a supergraph of the plant graph, the so-called dynamic deadbeat control design strategy is an undominated minimizer of the competitive ratio. For any fixed plant, the controller given by the deadbeat control design strategy can be separated into a static feedback law and a dynamic disturbance observer. For constant disturbances, it is shown that this structure corresponds to a proportional-integral controller. However, the deadbeat control design strategy is dominated when the plant graph has sinks. We present an undominated limited model information control design method that takes advantage of the knowledge of the sinks' location to achieve a better closed-loop performance. We further show that this control design strategy has the same competitive ratio as the deadbeat control design strategy. Later, we characterize the amount of model information needed to achieve a better competitive ratio than the deadbeat control design strategy. The amount of information is captured using the design graph, that is, a directed graph which indicates the dependency of each subcontroller on different parts of the global dynamical model. It turns out that, to achieve a better competitive ratio than the deadbeat control design strategy, each subsystem's controller should, at least, has access to the model of all those subsystems that can affect it.

\jpfig{0.7}{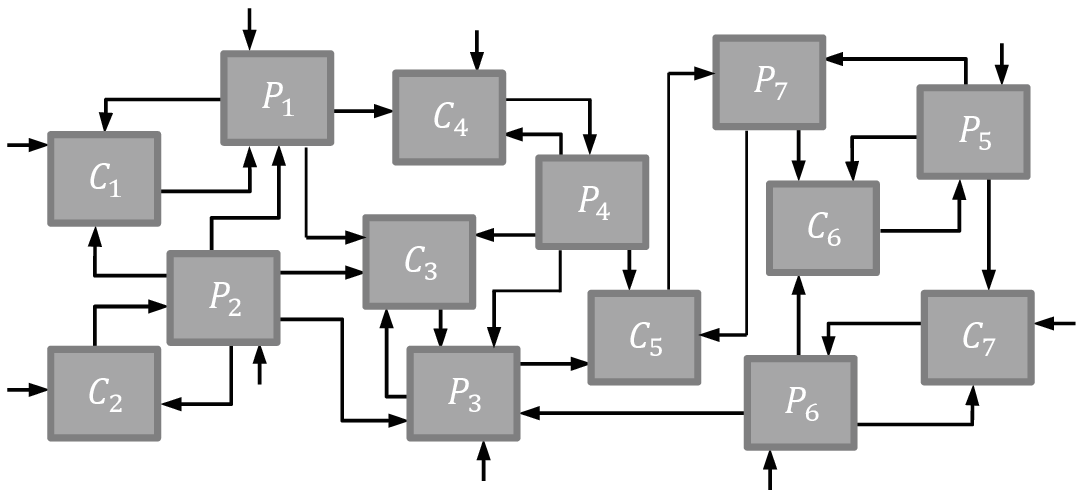}{FigureNCS}{ Illustrative example of a networked control system. }

This paper is organized as follows. We formulate the problem and define the performance metric in Section~\ref{sec_2}. In Section~\ref{sec_2_1}, we introduce two specific control design strategies and study their properties. We characterize the best limited model information control design method as a function of the subsystems interconnection pattern in Section~\ref{sec:Gp}. In Section~\ref{sec:Gc}, we study the influence of the amount of the information available to each subsystem on the quality of the controllers that they can produce. We discuss special cases of constant-disturbance rejection, step-reference tracking, and proportional-integral control in Section~\ref{sec_3.5}. Finally, we end with conclusions in Section~\ref{sec_4}.

\subsection{Notation}
The set of real numbers and complex numbers are denoted by $\mathbb{R}$ and $\mathbb{C}$, respectively. All other sets are denoted by calligraphic letters, such as $\p$ and $\A$. Particularly, the letter $\mathcal{R}$ denotes the set of proper real rational functions.

Matrices are denoted by capital roman letters such as $A$. $A_j$ will denote the $j^{\textrm{\scriptsize{th}}}$ row of $A$. $A_{ij}$ denotes a submatrix of matrix $A$, the dimension and the position of which will be defined in the text. The entry in the $i^{\textrm{\scriptsize{th}}}$ row and the $j^{\textrm{\scriptsize{th}}}$ column of the matrix $A$ is $a_{ij}$.

Let $\mathcal{S}_{++}^n$ ($\mathcal{S}_{+}^n$) be the set of symmetric positive definite (positive semidefinite) matrices in $\Re^{n\times n}$. $A > (\geq) 0$ means that the symmetric matrix $A\in \Re^{n\times n}$ is positive definite (positive semidefinite) and $A > (\geq) B$ means that $A-B > (\geq) 0$.

$\underline{\sigma}(Y)$ and $\overline{\sigma}(Y)$ denote the smallest and the largest singular values of the matrix $Y$, respectively. Vector $e_i$ denotes the column-vector with all entries zero except the $i^{\textrm{\scriptsize{th}}}$ entry, which is equal to one.

All graphs considered in this paper are directed, possibly with self-loops, with vertex set $\{1,...,q\}$ for some positive integer $q$. If $G=(\{1,...,q\},E)$ is a directed graph, we say that $i$ is a sink if there does not exist $j \neq i $ such that $(i,j) \in E$. The adjacency matrix $S\in\{0,1\}^{q\times q}$ of graph $G$ is a matrix whose entries are defined as $s_{ij}=1$ if $(j,i) \in E$ and $s_{ij}=0$ otherwise. Since the set of vertices is fixed for all considered graphs, a subgraph of a graph $G$ is a graph whose edge set is a subset of the edge set of $G$ and a supergraph of a graph $G$ is a graph of which $G$ is a subgraph. We use the notation $G'\supseteq G$ to indicate that $G'$ is a supergraph of $G$.

\section{Mathematical Formulation} \label{sec_2}
\subsection{Plant Model}
We are interested in discrete-time linear time-invariant dynamical systems described by
\begin{equation} \label{eqn_1}
x(k+1)=Ax(k)+B(u(k)+w(k)) \; ; \; x(0)=x_0,
\end{equation}
where $x(k)\in \mathbb{R}^{n}$ is the state vector, $u(k)\in \mathbb{R}^{n}$ is the control input, $w(k)\in \mathbb{R}^{n}$ is the disturbance vector and $A\in \mathbb{R}^{n\times n}$ and $B\in \mathbb{R}^{n\times n}$ are appropriate model matrices. Furthermore, we assume that the dynamic disturbance can be modeled as
\begin{equation} \label{eqn_2}
w(k+1)=Dw(k) \; ; \; w(0)=w_0,
\end{equation}
where $w_0\in\mathbb{R}^n$ is unknown to the controller (and the control designer). Let a plant graph $G_{\p}$ with adjacency matrix $S_{\p}$ be given. We define the following set of matrices
$$
\A(S_{\p})=\{ \bar{A} \in \mathbb{R}^{n \times n} \; | \; \bar{a}_{ij} = 0  \mbox{ for all } 1\leq i,j \leq n \mbox { such that } (s_{\p})_{ij}=0 \}.
$$
Also, let us define
$$
\B(\epsilon) =\{ \bar{B} \in \mathbb{R}^{n \times n} \; | \; \underline{\sigma}(\bar{B}) \geq \epsilon, \bar{b}_{ij} = 0 \mbox{ for all } 1 \leq i \neq j \leq n \},
$$
for some given scalar $\epsilon >0$ and
$$
\D=\{ \bar{D} \in \mathbb{R}^{n \times n} \; | \; \bar{d}_{ij} = 0 \mbox{ for all } 1 \leq i \neq j \leq n \}.
$$
Now, we can introduce the set of plants of interest $\p$ as the set of all discrete-time linear time-invariant systems~(\ref{eqn_1})--(\ref{eqn_2}) with $A \in \A(S_{\p})$, $B \in \B(\epsilon)$, $D \in \D$, $x_0 \in \mathbb{R}^n$ and $w_0 \in \mathbb{R}^n$. With a slight abuse of notation, we will henceforth identify a plant $P\in\p$ with its corresponding tuple $(A,B,D,x_0,w_0)$.

The variables $x_i\in\mathbb{R}$, $u_i\in\mathbb{R}$, and $w_i\in\mathbb{R}$ are the state, input, and disturbance of scalar subsystem $i$ whose dynamics are given by
\begin{eqnarray}
x_i(k+1) &=& \sum_{j=1}^n a_{ij} x_j(k) + b_{ii} (u_i(k)+w_i(k)), \nonumber \\
w_i(k+1) &=& d_{ii} w_i(k). \nonumber
\end{eqnarray}
We call $G_\p$ the plant graph since it illustrates the interconnection structure between different subsystems, that is, subsystem $j$ can affect subsystem $i$ only if $(j,i) \in E_{\p}$. Note that we assume that the global system is fully-actuated; i.e., all the matrices $B\in\B(\epsilon)$ are square invertible matrices. This assumption is motivated by the fact that we need all subsystems to be directly controllable. Moreover, we make the standing assumption that the plant graph $G_\p$ contain no isolated node. There is no loss of generality in assuming that there is no isolated node in the plant graph $G_{\mathcal{P}}$, since it is always possible to design a controller for an isolated subsystem without any model information about the other subsystems and without influencing the overall system performance. Note that, in particular, this implies that there are $q \geq 2$ vertices in the graph because for $q=1$ the only subsystem that exists is an isolated node in the plant graph.

Figure~\ref{figure0a}($a$) shows an example of a plant graph $G_\mathcal{P}$. Each node represents a subsystem of the system. For instance, the second subsystem in this example affects the first subsystem and the third subsystem, that is, submatrices $A_{12}$ and $A_{32}$ can be nonzero. Note that the first subsystem in Figure~\ref{figure0a}($a$) represents a sink of $G_\p$. The plant graph $G'_\p$ in Figure~\ref{figure0a}($a'$) has no sink.

\jpfig{0.62}{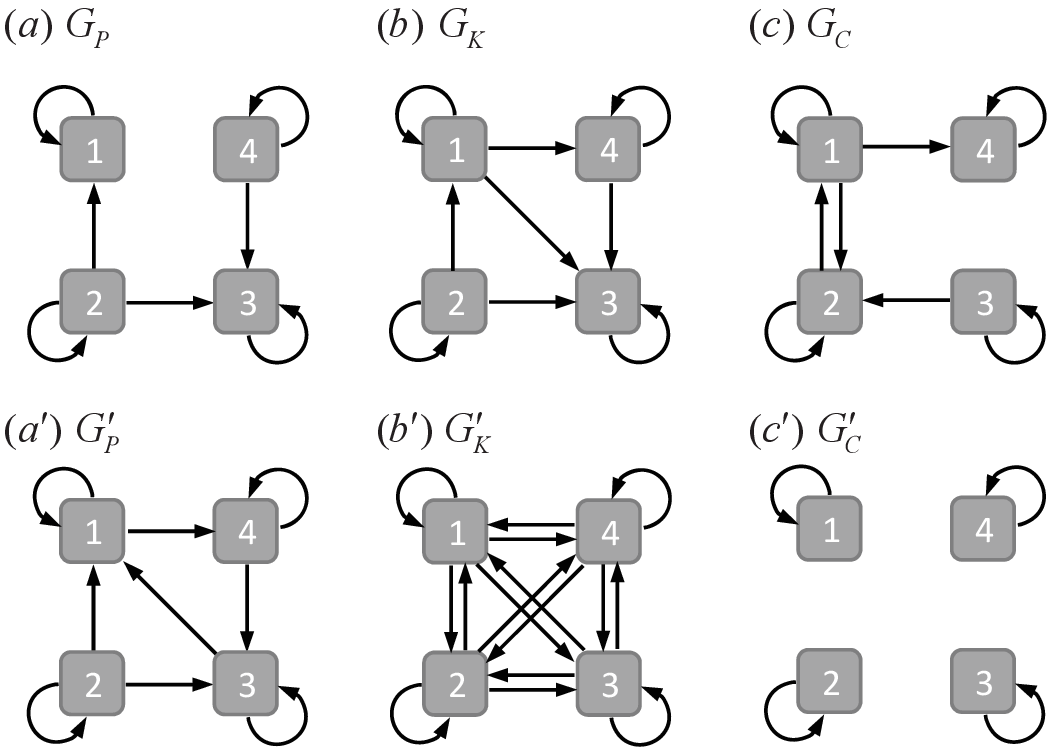}{figure0a}{ $G_\p$ and $G'_\p$ are examples of plant graphs, $G_\K$ and $G'_\K$ are examples of control graphs, and $G_\comm$ and $G'_\comm$ are examples of design graphs.}

\subsection{Controller Model} \label{subsec:controller}
The control laws of interest in this paper are discrete-time linear time-invariant dynamic state-feedback control laws of the form
\begin{eqnarray} \label{eqn:controller}
x_K(k+1)&=&A_K x_K(k) + B_K x(k) \; ; \; x_K(0)=0, \\
u(k)&=&C_K x_K(k)+D_K x(k).
\end{eqnarray}
Each controller can also be represented by a transfer function
$$
K\triangleq\left[\begin{array}{c|c} A_K & B_K \\ \hline C_K & D_K \end{array}\right]=C_K(zI-A_K)^{-1}B_K+D_K,
$$
where $z$ is the symbol for the one time-step forward shift operator. Let a control graph $G_{\K}$ with adjacency matrix $S_{\K}$ be given. Each controller $K$ belongs to
$$
\K(S_{\K})=\{ K \in  \; \R^{n \times n} \;|\; k_{ij} = 0 \mbox{ for all }  1\leq i,j \leq n  \mbox { such that } (s_{\K})_{ij}=0 \}.
$$
When the adjacency matrix $S_{\K}$ is not relevant or can be deduced from context, we refer to the set of controllers as $\K$. Since it makes sense for each subcontroller to use at least its corresponding subsystem state-measurements, we make the standing assumption that in each design graph $G_\K$, all the self-loops are present.

An example of a control graph $G_\K$ is given in Figure~\ref{figure0a}($b$). Each node represents a subsystem--controller pair of the overall system. For instance, $G_\K$ shows that the first subcontroller can use state measurements of the second subsystem beside its corresponding subsystem state-measurements. Figure~\ref{figure0a}($b'$) shows a complete control graph $G'_{\mathcal{K}}$. This control graph indicates that each subcontroller has access to full state measurements of all subsystems, that is, $\mathcal{K}(S_{\mathcal{K}})=\R^{n \times n}$.

\subsection{Control Design Methods}
A control design method $\Gamma$ is a map from the set of plants $\p$ to the set of controllers $\K$. Any control design method $\Gamma$ has the form
\begin{equation}
\Gamma=\matrix{ccc}{ \gamma_{11} & \cdots & \gamma_{1n} \\ \vdots & \ddots & \vdots \\ \gamma_{n1} & \cdots & \gamma_{nn} },
\label{eq_gamma}
\end{equation}
where each entry $\gamma_{ij}$ represents a map $\A(S_{\p}) \times \B(\epsilon) \times \D  \rightarrow \R$.

Let a design graph $G_{\mathcal{C}}$ with adjacency matrix $S_{\mathcal{C}}$ be given. We say that $\Gamma$ has structure $G_{\mathcal{C}}$, if for all $i$, subcontroller $i$ is computed with knowledge of the plant model of only those subsystems $j$ such that $(j,i) \in E_{\mathcal{C}}$. Equivalently, $\Gamma$ has structure $G_\comm$, if for all $i$, the map $\Gamma_i= [\gamma_{i1}\; \cdots\; \gamma_{in}]$ is only a function of $\{ [a_{j1}\;\cdots\;a_{jn}] , b_{jj}, d_{jj} \; | \; (s_{\mathcal{C}})_{ij} \neq 0\}$. When $G_{\mathcal{C}}$ is not a complete graph, we refer to $\Gamma \in \mathcal{C}$ as being a ``limited model information control design method''. Since it makes sense for the designer of each subcontroller to have access to at least its corresponding subsystem model parameters, we make the standing assumption that in each design graph $G_\comm$, all the self-loops are present.

\par The set of all control design strategies with structure $G_{\mathcal{C}}$ will be denoted by $\mathcal{C}$, which is considered as a subset of all maps from $\A(S_{\p}) \times \B(\epsilon) \times \D$ to $\K(S_{\K})$ because a design method with structure $G_{\mathcal{C}}$ is not a function of the initial state $x_0$ or the initial disturbance $w_0$. We use the notation $\Gamma(A,B,D)$ instead of $\Gamma(P)$ for each plant $P=(A,B,D,x_0,w_0) \in \p$ to emphasize this fact.

To simplify the notation, we assume that any control design strategy $\Gamma$ has a state-space realization of the form
$$
\Gamma(A,B,D)=\left[\begin{array}{c|c} A_\Gamma(A,B,D) & B_\Gamma(A,B,D) \\ \hline C_\Gamma(A,B,D) & D_\Gamma(A,B,D) \end{array}\right],
$$
where $A_\Gamma(A,B,D)$, $B_\Gamma(A,B,D)$, $C_\Gamma(A,B,D)$, and $D_\Gamma(A,B,D)$ are matrices of appropriate dimension for each plant $P=(A,B,D,x_0,w_0)\in \p$. The matrices \linebreak[4] $A_\Gamma(A,B,D)$ and $C_\Gamma(A,B,D)$ are block diagonal matrices since subcontrollers do not share state variables. This realization is not necessarily minimal.

An example of a design graph $G_{\mathcal{C}}$ is given in Figure~\ref{figure0a}($c$). Each node represents a subsystem--controller pair of the overall system. For instance, $G_\comm$ shows that the second subsystem's model is available to the designer of the first subsystem's controller but not the third and the forth subsystems' model. Figure~\ref{figure0a}($c'$) shows a fully disconnected design graph $G'_\comm$. A local designer in this case can only rely on the model of its corresponding subsystem.

\subsection{Performance Metric}
The goal of this paper is to investigate the influence of the plant graph on the properties of controllers derived from limited model information control design methods. We use two performance metrics to compare different control design methods, which are adapted from the notions of competitive ratio and domination recently introduced in~\cite{Langbort2010,Farokhi_ACC_2011,FarokhiLangbortJohansson2011,Farokhi2011}. Let us start with introducing the closed-loop performance criterion.

\par To each plant $P=(A,B,D,x_0,w_0) \in \p$ and controller $K\in \K$, we associate the performance criterion
\begin{equation} \label{eqn:1}
J_P (K)=\sum_{k=0}^\infty \left[ x(k)^TQx(k) + (u(k)+w(k))^TR(u(k)+w(k)) \right],
\end{equation}
where $Q \in \mathcal{S}_{++}^n$ and $R \in \mathcal{S}_{++}^n$ are diagonal matrices. We make the following standing assumption:

\begin{assumption} $Q=R=I$. \end{assumption}

This is without loss of generality because the change of variables $(\bar{x},\bar{u},\bar{w})= (Q^{1/2} x,R^{1/2}u,R^{1/2}w)$ transforms the closed-loop performance measure and state-space representation
into
\begin{equation}
\label{cost_easy}
J_P (K)=\sum_{k=0}^\infty \left[ \bar{x}(k)^T\bar{x}(k) + (\bar{u}(k)+\bar{w}(k))^T(\bar{u}(k)+\bar{w}(k)) \right],
\end{equation}
and
\begin{eqnarray}
\bar{x}(k+1)&=&Q^{1/2}AQ^{-1/2}\bar{x}(k)+Q^{1/2}BR^{-1/2}(\bar{u}(k)+\bar{w}(k))\nonumber \\ &=&\bar{A}\bar{x}(k)+\bar{B}(\bar{u}(k)+\bar{w}(k)),\nonumber
\end{eqnarray}
without affecting the plant, control, or design graphs, due to $Q$ and $R$ being diagonal matrices.

\begin{definition}\emph{(Competitive Ratio)} Let a plant graph $G_{\p}$, a control graph $G_{\K}$, and a constant $\epsilon > 0$ be given. Assume that, for every plant $P \in \p$, there exists an optimal controller $K^*(P) \in \K$ such that
$$
J_P (K^*(P))\leq J_P (K), \; \forall K \in \K.
$$
The competitive ratio of a control design method $\Gamma$ is defined as
$$
r_{\p} (\Gamma)= \sup_{P=(A,B,D,x_0,w_0) \in \p} \frac{J_P (\Gamma(A,B,D))}{J_P (K^*(P))},
$$
with the convention that ``$\frac{0}{0}$'' equals one.
\label{def_comp_rat}
\end{definition}

Note that the optimal control design strategy (with full plant model information) $K^*$ does not necessarily belong to the set $\comm$.

\begin{definition}\emph{(Domination)} A control design method $\Gamma$ is said to dominate another control design method $\Gamma'$ if
\begin{equation}
J_P(\Gamma(A,B,D))\leq J_P(\Gamma'(A,B,D)),\hspace{0.1in} \forall \; P=(A,B,D,x_0,w_0)\in \p,
\label{comp}
\end{equation}
with strict inequality holding for at least one plant in $\p$. When $\Gamma' \in \mathcal{C}$ and no control design method $\Gamma \in \mathcal{C}$ exists that satisfies (\ref{comp}), we say that $\Gamma'$ is undominated in $\mathcal{C}$ for plants in $\p$.
\end{definition}

In the remainder of this paper, we determine optimal control design strategies
\begin{equation} \label{eqn:0}
\Gamma^*\in\argmin_{\Gamma \in \comm} r_{\p} (\Gamma),
\end{equation}
for a given plant, control, and design graph. Since several design methods may achieve this minimum, we are interested in determining which ones of these strategies are undominated.

\section{Preliminary Results} \label{sec_2_1}
Before stating the main results of the paper, we introduce two specific control design strategies and study their properties.

\subsection{Optimal Centralized Control Design Strategy} \label{subsec_OCCDS}
The problem of designing optimal constant input-disturbance accommodation control for linear time-invariant continuous-time systems was solved earlier in~\cite{Anderson1971,Johnson1968}. To the best of our knowledge, this was not the case for arbitrary dynamic disturbance accommodation when dealing with linear time-invariant discrete-time systems. As we need it later, we start by developing the optimal centralized (i.e, $G_\K$ is a complete graph) disturbance accommodation controller $K^*(P)$ for a given plant $P\in \p$. First, let us define the auxiliary variables $\xi(k)=u(k)+w(k)$ and $\bar{u}(k)=u(k+1)-Du(k)$. It then follows that
\begin{eqnarray}
\xi(k+1)&=&u(k+1)+w(k+1) \nonumber \\&=&u(k+1)+D w(k) \nonumber \\&=&Du(k)+D w(k)+\bar{u}(k) \nonumber \\&=&D\xi(k)+\bar{u}(k). \label{eqn:xi}
\end{eqnarray}
Augmenting the state-transition in~(\ref{eqn:xi}) with the state-space representation of the system in~(\ref{eqn_1}) results in
\begin{eqnarray} \label{eqn:new_plant}
\matrix{c}{x(k+1)\\ \xi(k+1)}=\matrix{cc}{A & B \\0 & D}\matrix{c}{x(k)\\ \xi(k)}+\matrix{c}{0\\ I}\bar{u}(k).
\end{eqnarray}
Besides, we can write the performance measure in~(\ref{cost_easy}) as
\begin{eqnarray} \label{eqn:new_cost}
J_P (K)=\sum_{k=0}^\infty \matrix{c}{x(k)\\ \xi(k)}^T\matrix{c}{x(k)\\ \xi(k)}.
\end{eqnarray}
To guarantee the existence and uniqueness of the optimal controller $K^*(P)$, we need the following lemma.

\begin{lemma} \label{prop:0} The pair $(\tilde{A},\tilde{B})$, with
\begin{equation} \label{ABtilde}
\tilde{A}=\matrix{cc}{A & B \\0 & D},\; \tilde{B}=\matrix{c}{0\\ I},
\end{equation}
is controllable for any given $P=(A,B,D,x_0,w_0)\in\p$. \end{lemma}

\begin{proof} The pair $(\tilde{A},\tilde{B})$ is controllable if and only if
$$
\matrix{c;{2pt/2pt}c}{\tilde{A}-\lambda I & \tilde{B}}=\matrix{cc;{2pt/2pt}c}{A-\lambda I & B & 0 \\0 & D-\lambda I & I }
$$
is full-rank for all $\lambda\in \mathbb{C}$. This condition is always satisfied since all matrices $B\in\B(\epsilon)$ are full-rank matrices.
\end{proof}

Now the problem of minimizing the cost function in~(\ref{eqn:new_cost}) subject to plant dynamics in~(\ref{eqn:new_plant}) becomes a state-feedback linear quadratic optimal control with a unique solution of the form
$$
\bar{u}(k)=G_1 x(k) + G_2 \xi(k),
$$
where $G_1\in \mathbb{R}^{n\times n}$ and $G_2 \in \mathbb{R}^{n\times n}$ satisfy
\begin{equation} \label{eqn:G}
\matrix{cc}{G_1 & G_2}=-(\tilde{B}^TX\tilde{B})^{-1}\tilde{B}^TX\tilde{A}
\end{equation}
and $X$ is the unique positive-definite solution of the discrete algebraic Riccati equation
\begin{equation} \label{eqn_Riccati}
\tilde{A}^TX\tilde{B}(\tilde{B}^TX\tilde{B})^{-1}\tilde{B}^TX\tilde{A}-\tilde{A}^TX\tilde{A}+X-I=0.
\end{equation}
Therefore, we have
\begin{eqnarray}
u(k+1)&=&Du(k)+\bar{u}(k) \nonumber \\&=&Du(k)+G_1 x(k)+G_2 \xi(k). \label{eqn:uk+11}
\end{eqnarray}
Using the identity $\xi(k)=B^{-1}(x(k+1)-Ax(k))$ in~(\ref{eqn:uk+11}), we get
\begin{eqnarray}\label{eqn:uk+1}
u(k+1)&=&Du(k)+G_1 x(k)+G_2 \xi(k) \nonumber \\&=&Du(k)+G_1 x(k)+G_2 B^{-1}(x(k+1)-Ax(k)) \nonumber \\&=&Du(k)+(G_1-G_2B^{-1}A) x(k)+G_2 B^{-1}x(k+1).
\end{eqnarray}
Putting a control signal of the form $u(k)=x_K(k)+D_K x(k)$ in~(\ref{eqn:uk+1}), we get
\begin{eqnarray}
x_K(k+1)=Dx_K(k)+(DD_K+G_1-G_2 B^{-1}A)x(k)+(G_2 B^{-1}-D_K)x(k+1). \nonumber
\end{eqnarray}
Now, we enforce the condition $G_2 B^{-1}-D_K=0$, as $x_K(k+1)$ can only be a function of $x(k)$ and $x_K(k)$, see~(\ref{eqn:controller}). Therefore, the optimal controller $K^*(P)$ becomes
\begin{eqnarray}
x_K(k+1)&=&Dx_K(k)+[G_1 +DG_2 B^{-1}-G_2 B^{-1}A]x(k), \nonumber \\
u(k)&=&x_K(k)+G_2 B^{-1} x(k), \nonumber
\end{eqnarray}
with $x_K(0)=0$.

\begin{lemma} \label{lem:1}
Let the control graph $G_{\K}$ be a complete graph. Then, the cost of the optimal controller $K^*(P)$ for each plant $P\in \p$ is lower-bounded as
$$
J_P(K^*(P))\geq \matrix{c}{x_0 \\ Bw_0}^T \matrix{cc}{W+DWD+D^2B^{-2} & -D(W+B^{-2}) \\ -(W+B^{-2})D & W+B^{-2} } \matrix{c}{x_0 \\ Bw_0},
$$
where
$$
W=A^T(I+B^2)^{-1}A+I.
$$
\end{lemma}

\begin{proof} Define
$$
\bar{J}_P (K,\rho)=\sum_{k=0}^\infty \left( \matrix{c}{x(k)\\ \xi(k)}^T\matrix{c}{x(k)\\ \xi(k)} + \rho \bar{u}(k)^T\bar{u}(k) \right),
$$
and
$$
\bar{K}_\rho^*(P)=\argmin_{K\in\K} \bar{J}_P (K,\rho).
$$
Using Lemma~\ref{prop:0}, we know that $\bar{K}_\rho^*(P)$ exists and is unique. We can find $\bar{J}_P (\bar{K}_\rho^*(P),\rho)$ using $X(\rho)$ as the unique positive definite solution of the discrete algebraic Riccati equation
\begin{equation} \label{eqn_Riccati_rho}
\tilde{A}^TX(\rho)\tilde{B}(\rho I+\tilde{B}^TX(\rho)\tilde{B})^{-1}\tilde{B}^TX(\rho)\tilde{A}-\tilde{A}^TX(\rho)\tilde{A}+X(\rho)-I=0.
\end{equation}
According to \cite{Komaroff1994}, the positive-definite matrix $X(\rho)$ is lower-bounded by
\begin{eqnarray}
X(\rho)-I &\geq& \tilde{A}^T\left(\bar{X}(\rho)^{-1}+\rho^{-1}\tilde{B}\tilde{B}^T\right)^{-1}\tilde{A} \nonumber\\ &=&\tilde{A}^T\left(\bar{X}(\rho)-\bar{X}(\rho)\tilde{B}\left(\rho I+\tilde{B}^T\bar{X}(\rho)\tilde{B}\right)^{-1}\tilde{B}^T\bar{X}(\rho)\right)\tilde{A}, \nonumber
\end{eqnarray}
where
\begin{eqnarray}
\bar{X}(\rho)&=&\tilde{A}^T\left(I+\rho^{-1}\tilde{B}\tilde{B}^T\right)^{-1}\tilde{A} \nonumber = \matrix{cc}{A^TA+I & A^TB \\ BA & B^2+D^2\frac{\rho}{\rho+1}+I}.\nonumber
\end{eqnarray}
Basic algebraic calculations show that
\begin{eqnarray}
\lim_{\rho\rightarrow 0} \left[\bar{X}(\rho)-\bar{X}(\rho)\tilde{B}(\rho I+\tilde{B}^T\bar{X}(\rho)\tilde{B})^{-1}\tilde{B}^T\bar{X}(\rho)\right]\hspace{-.07in} =\hspace{-.07in}\matrix{cc}{A^T(I+B^2)^{-1}A+I & 0 \\ 0 & 0 }.\nonumber
\end{eqnarray}
According to~\cite{Kondo1986}, we know that
$$
\lim_{\rho\rightarrow 0^+}\bar{J}_P (\bar{K}_\rho^*(P),\rho)=J_P(K^*(P)),
$$
and as a result
\begin{eqnarray} \label{eqn_16}
X=\lim_{\rho \rightarrow 0}X(\rho) \geq \matrix{cc}{A & B \\0 & D}^T\matrix{cc}{A^T(I+B^2)^{-1}A+I & 0 \\ 0 & 0 }\matrix{cc}{A & B \\0 & D}+I. \hspace{.3in}
\end{eqnarray}
where $X$ is the unique positive-definite solution of the discrete algebraic Riccati equation in~(\ref{eqn_Riccati}) and consequently
$$
J_P(K^*(P))=\matrix{c}{x_0 \\ \xi(0)}^T \matrix{cc}{X_{11} & X_{12} \\ X_{12}^T & X_{22} } \matrix{c}{x_0 \\ \xi(0)}
$$
with $X$ being partitioned as
$$
X=\matrix{cc}{X_{11} & X_{12} \\ X_{12}^T & X_{22} }.
$$
We know that
$$
\xi(0)=u(0)+w_0=G_2 B^{-1}x_0+w_0=-(X_{22}^{-1}X_{12}^T+DB^{-1})x_0+w_0.
$$
Thus, the cost of the optimal control design $J_P(K^*(P))$ becomes
\begin{eqnarray}
&&\hspace{-.3in}\matrix{c}{x_0 \\ -(X_{22}^{-1}X_{12}^T+DB^{-1})x_0+w_0}^T \hspace{-.06in}\matrix{cc}{X_{11} & X_{12} \\ X_{12}^T & X_{22} } \hspace{-.06in}\matrix{c}{x_0 \\ -(X_{22}^{-1}X_{12}^T+DB^{-1})x_0+w_0}\nonumber \\
&& \hspace{-0.1in}=\matrix{c}{x_0 \\ w_0}^T \matrix{cc}{X_{11}-X_{12}X_{22}^{-1}X_{12}^T+B^{-1}DX_{22}DB^{-1} & -B^{-1}DX_{22} \\ -X_{22}DB^{-1} & X_{22} } \matrix{c}{x_0 \\ w_0}\nonumber \\ && \hspace{-0.1in}=\matrix{c}{x_0 \\ w_0}^T \matrix{cc}{B^{-1}(X_{22}+DX_{22}D-I)B^{-1} & -B^{-1}DX_{22} \\ -X_{22}DB^{-1} & X_{22} } \matrix{c}{x_0 \\ w_0} \label{eqn:inner}
\end{eqnarray}
The second equality is true because of the following equation extracted from the discrete algebraic Riccati equation in~(\ref{eqn_Riccati})
$$
X_{22}=I+BX_{11}B-BX_{12}X_{22}^{-1}X_{12}^TB,
$$
which is equivalent to
\begin{equation} \label{eqn:subRiccati}
X_{11}-X_{12}X_{22}^{-1}X_{12}^T=B^{-1}(X_{22}-I)B^{-1}.
\end{equation}
Using~(\ref{eqn_16}), it is evident that
$$
X_{22}\geq B[A^T(I+B^2)^{-1}A+I]B+I=BWB+I,
$$
and as a result, the inner-matrix in~(\ref{eqn:inner}) is lower-bounded by
\begin{eqnarray}
&&\hspace{-.2in}\matrix{cc}{B^{-1}(X_{22}+DX_{22}D-I)B^{-1} & -B^{-1}DX_{22} \\ -X_{22}DB^{-1} & X_{22} } \nonumber\\&&\hspace{0.3in}=\matrix{cc}{B^{-1}(X_{22}-I)B^{-1} & 0 \\ 0 & 0 }+\matrix{cc}{B^{-1}DX_{22}DB^{-1} & -B^{-1}DX_{22} \\ -X_{22}DB^{-1} & X_{22} } \nonumber\\&&\hspace{0.3in}=\matrix{cc}{B^{-1}(X_{22}-I)B^{-1} & 0 \\ 0 & 0 }+\matrix{c}{-B^{-1}D \\ I}X_{22}\matrix{c}{-B^{-1}D \\ I}^T \nonumber \\ && \hspace{0.3in} \geq  \matrix{cc}{B^{-1}(B W B)B^{-1} & 0 \\ 0 & 0 }+\matrix{c}{-B^{-1}D \\ I}(BWB+I)\matrix{c}{-B^{-1}D \\ I}^T \nonumber \\ && \hspace{0.3in} = \matrix{cc}{W+DWD+D^2B^{-2} & -D(WB+B^{-1}) \\ -(BW+B^{-1})D & BWB+I } \nonumber
\end{eqnarray}
Finally, we get
\begin{equation}
\begin{split}
J_P(K^*(P))&\geq \matrix{c}{x_0 \\ w_0}^T \matrix{cc}{W+DWD+D^2B^{-2} & -D(WB+B^{-1}) \\ -(BW+B^{-1})D & BWB+I } \matrix{c}{x_0 \\ w_0}\nonumber \\&=\matrix{c}{x_0 \\ Bw_0}^T \matrix{cc}{W+DWD+D^2B^{-2} & -D(W+B^{-2}) \\ -(W+B^{-2})D & W+B^{-2} } \matrix{c}{x_0 \\ Bw_0} \nonumber.
\end{split}
\end{equation}
This statement concludes the proof.
\end{proof}

\subsection{Deadbeat Control Design Strategy}
In this subsection, we introduce the deadbeat control design strategy and calculate its competitive ratio.

\begin{definition} \label{def:1} The deadbeat control design strategy $\Gamma^\Delta: \A(S_\p) \times \B(\epsilon) \times \D  \rightarrow \K$ is defined as
$$
\Gamma^\Delta(A,B,D)\triangleq\left[\begin{array}{c|c} D & -B^{-1}D^2 \\ \hline I & -B^{-1}(A+D) \end{array}\right].
$$
\end{definition}

It should be noted that using the deadbeat control design strategy, irrespective of the value of the initial state $x_0$ and the initial disturbance $w_0$, the closed-loop system reaches the origin in just two time-steps. The closed-loop system with deadbeat control design strategy is shown in Figure~\ref{figure1}(a). This feedback loop can be rearranged as the one in Figure~\ref{figure1}(b) which has two separate components. One component is a static deadbeat control design strategy for regulating the state of the plant and the other one is a deadbeat observer for canceling the disturbance. This structure is further discussed in Section~\ref{sec_3.5}, where it is shown that it corresponds to proportional-integral control in some cases. First, we need to calculate an expression for the cost of the deadbeat control design strategy.

\begin{lemma} \label{lem:delta} The cost of the deadbeat control design strategy $\Gamma^\Delta$ for each plant $P=(A,B,D,x_0,w_0)\in \p$ is
$$
J_P (\Gamma^\Delta(A,B,D))=\matrix{c}{x_0 \\ Bw_0}^T \matrix{cc}{Q_{11} & Q_{12}  \\ Q_{12}^T & Q_{22}} \matrix{c}{x_0 \\ Bw_0},
$$
where
\begin{eqnarray}
\label{eqn_15.1}Q_{11}&\hspace{-0.1in}=&\hspace{-0.1in}I+D^2(I+B^{-2})+A^TB^{-2}A+DA^TB^{-2}AD+A^TB^{-2}D+DB^{-2}A, \hspace{.2in} \\
\label{eqn_15.2}Q_{12}&\hspace{-0.1in}=&\hspace{-0.1in}-D-A^TB^{-2}-DB^{-2}-DA^TB^{-2}A, \\
\label{eqn_15.3}Q_{22}&\hspace{-0.1in}=&\hspace{-0.1in}A^TB^{-2}A+B^{-2}+I.
\end{eqnarray}
\end{lemma}
\begin{proof} First, it should be noted that the state of the closed-loop system with $\Gamma^\Delta(A,B,D)$ in feedback reaches the origin in two time-steps. Now, using the system state transition, one can calculate the deadbeat control design strategy cost as
\begin{equation}
\begin{split}
J_P (\Gamma^\Delta(A,B,D))&= x_0^Tx_0+(u(0)+w_0)^T(u(0)+w_0) \nonumber \\ &\hspace{.8in}+x(1)^Tx(1)+(u(1)+w(1))^T(u(1)+w(1)), \nonumber
\end{split}
\end{equation}
where $x(1)=-Dx_0+Bw_0$, $u(0)=-B^{-1}(A+D)x_0,$ and $u(1)=-B^{-1}(A+D)x(1)-B^{-1}D^2x_0$. The rest of the proof is a trivial simplification.
\end{proof}

\jpfig{1.0}{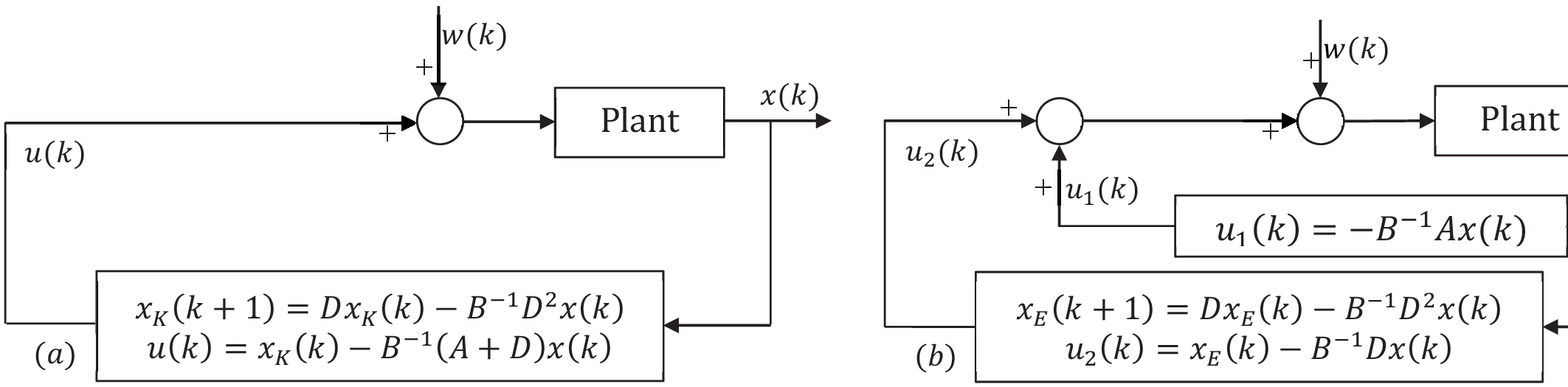}{figure1}{The closed-loop system with (a)~the deadbeat control design strategy $\Gamma^\Delta$, and (b)~rearranging this control design strategy as a static deadbeat control design and a deadbeat observer design.}

We need the following lemma in order to calculate the competitive ratio of the deadbeat control design strategy $\Gamma^\Delta$ when the control graph $G_\K$ is a supergraph of the plant graph $G_\p$. As the notation $K^*(P)$ is reserved for the optimal control design strategy for a given control graph $G_\K$, from now on, we will use $K^*_C$ to denote the centralized optimal control design strategy (i.e., the optimal control design strategy with access to full-state measurement).

\begin{lemma} \label{prop:2} Let $G_{\K}\supseteq G_\p$, and $P=(A,B,D,x_0,w_0)\in \p$ be a plant with $A$ being a nilpotent matrix of degree two. Then, $J_P(K^*(P))=J_P(K^*_C(P))$. \end{lemma}

\begin{proof}
When matrix $A$ is nilpotent, the unique positive-definite solution of the discrete algebraic Riccati equation~(\ref{eqn_Riccati}) is
$$
X=\matrix{cc}{A^TA+I & A^TB \\ BA & BA^T(I+B^2)^{-1}AB+I+B^2 }.
$$
Consequently, the optimal centralized controller gains in~(\ref{eqn:G}) are
$$
G_1=0, \hspace{.2in} G_2=-(I+B^2)^{-1}BAB-D,
$$
and as a result, the optimal centralized controller $K^*_C(P)$ is
\begin{equation}
\begin{split}
K^*_C(P)&= \left[\begin{array}{c|c} D & D(I+B^2)^{-1}B^{-1}A-B^{-1}D^2 \\ \hline I  & -(I+B^2)^{-1}BA-B^{-1}D \end{array}\right] \nonumber \\ &= (zI-D)^{-1} D(I+B^2)^{-1}B^{-1}A-B^{-1}D^2 -(I+B^2)^{-1}BA-B^{-1}D. \nonumber
\end{split}
\end{equation}
Thus, $K^*_C(P)\in \K(S_\K)$ because the control graph $G_\K$ is a supergraph of the plant graph $G_\p$. Now, considering that $K^*(P)$ is the global optimal decentralized controller, it has a lower cost than any other decentralized controller $K\in \K(S_\K)$, specially $K^*_C(P)\in\K(S_\K)$ for this particular plant. Hence,
\begin{equation} \label{eqn:ineq1}
J_P(K^*(P))\leq J_P(K^*_C(P)).
\end{equation}
On the other hand, it is evident that
\begin{equation} \label{eqn:ineq2}
J_P(K^*_C(P)) \leq J_P(K^*(P)).
\end{equation}
This concludes the proof.
\end{proof}

\begin{remark} Finding the optimal structured controller is intractable in general, even when the global model is known. In this paper, we concentrate on the cases where the control graph $G_{\K}$ is a supergraph of the plant graph $G_\p$, because it is relatively easier to solve the optimal control design problem under limited model information in this case. In addition, although, in this paper, we may not be able to find the optimal structured controller $K^*(P)$ for a particular plant in some of the cases, we can still compute the competitive ratio $r_\p$. Thus, in a sense, this makes the competitive ratio a quite powerful tool.
\end{remark}

Next, we derive the competitive ratio of the deadbeat control design method.

\begin{theorem} \label{tho:1} Let $G_{\K}\supseteq G_\p$. Then, the competitive ratio of the deadbeat control design method $\Gamma^\Delta$ is equal to
$$
r_\p(\Gamma^\Delta)= \frac{2\epsilon^2+1+\sqrt{4\epsilon^2+1}}{2\epsilon^2}.
$$
\end{theorem}

\begin{proof} First, let us define the set of all real numbers that are greater than or equal to the competitive ratio of the deadbeat control design strategy
$$
\mathcal{M}=\left\{\beta\in\mathbb{R}\;\left|\; \frac{J_P(\Gamma^\Delta(A,B,D))}{J_P(K^*(P))}\leq \beta \right.\forall P\in\p \right\}.
$$
It is evident that
$$
J_P(K^*_C(P)) \leq J_P(K^*(P))
$$
for each plant $P\in\p$ irrespective of the control graph $G_\K$, and as a result
\begin{equation} \label{eqn:ineqKp*andKP*C}
\frac{J_P(\Gamma^\Delta(A,B,D))}{J_P(K^*(P))} \leq \frac{J_P(\Gamma^\Delta(A,B,D))}{J_P(K^*_C(P))}.
\end{equation}
Using~(\ref{eqn:ineqKp*andKP*C}) and Lemmas~\ref{lem:delta} and~\ref{lem:1}, $\beta$ belongs to the set $\mathcal{M}$ if
\begin{equation} \label{eqn_first_condition}
\hspace{0.2in} \frac{\matrix{c}{x_0 \\ Bw_0 }^T \matrix{cc}{Q_{11} & Q_{12}  \\ Q_{12}^T & Q_{22}} \matrix{c}{x_0 \\ Bw_0 }}{\matrix{c}{x_0 \\ Bw_0 }^T \matrix{cc}{W+DWD+D^2B^{-2} & -D(W+B^{-2}) \\ -(W+B^{-2})D & W+B^{-2} }  \matrix{c}{x_0 \\ Bw_0 }} \leq \beta,
\end{equation}
for all $A\in \A(S_\p)$, $B\in \B(\epsilon)$, $D\in \D$, $x_0\in \mathbb{R}^n$, and $w_0\in \mathbb{R}^n$ where $Q_{11}$, $Q_{12}$, and $Q_{22}$ are matrices defined in~(\ref{eqn_15.1})--(\ref{eqn_15.3}). The condition~(\ref{eqn_first_condition}) is satisfied, if and only if, for all $A\in \A(S_\p)$, $B\in \B(\epsilon)$, and $D\in \D $, we have
$$
\matrix{cc}{\beta(W+DWD+D^2B^{-2})-Q_{11} & -\beta D(W+B^{-2})-Q_{12} \\ -\beta (W+B^{-2})D-Q_{12}^T & \beta (W+B^{-2})-Q_{22} } \geq 0.
$$
Using Schur complement \cite{Zhang2005}, $\beta$ belongs to the set $\mathcal{M}$ if
\begin{eqnarray}
Z &=&\beta (W+B^{-2})-Q_{22}\nonumber \\&=&\beta (A^T(I+B^2)^{-1}A+I+B^{-2})-A^TB^{-2}A-B^{-2}-I \label{eqn:Z} \\&=&A^T(\beta (I+B^2)^{-1}-B^{-2})A+(\beta-1)(B^{-2}+I)\geq 0 \nonumber,
\end{eqnarray}
and
\begin{eqnarray} \label{eqn:condition1}
&&\hspace{-.4in}-\left[-\beta D(W+B^{-2})-Q_{12} \right] \left[\beta(W+B^{-2})-Q_{22}\right]^{-1}\left[-\beta (W+B^{-2})D-Q_{12}^T\right]\nonumber \\ && \hspace{1.5in} +\beta(W+DWD+D^2B^{-2})-Q_{11} \geq 0,
\end{eqnarray}
for all $A\in\A(S_\p)$, $B\in\B(\epsilon)$, and $D\in\D$. We can do the simplification
\begin{eqnarray}
-\beta D(W+B^{-2})-Q_{12}&=&-\beta D(A^T(I+B^2)^{-1}A+I+B^{-2})\nonumber \\  &&\hspace{0.4in}-(-D-A^TB^{-2}-DB^{-2}-DA^TB^{-2}A)\nonumber \\&=&-(\beta-1)D(I+B^{-2})+A^TB^{-2}\nonumber \\  &&\hspace{0.4in} -DA^T(\beta(I+B^2)^{-1}-B^{-2})A \nonumber \\&=&-DZ+A^TB^{-2}, \nonumber
\end{eqnarray}
and as a result, the condition~(\ref{eqn:condition1}) is equivalent to
\begin{equation} \label{eqn:condition1.5}
\beta(W+DWD+D^2B^{-2})-Q_{11}-[-DZ+A^TB^{-2}] Z^{-1}[-ZD+B^{-2}A]\geq 0,
\end{equation}
where $Z$ is defined in~(\ref{eqn:Z}). Furthermore, we can simplify $\beta(W+DWD+D^2B^{-2})-Q_{11}$ as
\begin{eqnarray}
&&A^T(\beta(I+B^2)^{-1}-B^{-2})A +(\beta-1)[ I+D^2B^{-2}+D^2 ]\nonumber \\ &&\hspace{1.0in}+DA^T(\beta(I+B^2)^{-1}-B^{-2})AD -A^TB^{-2}D-DB^{-2}A\nonumber,
\end{eqnarray}
which helps us to expand condition~(\ref{eqn:condition1.5}) to
\begin{eqnarray} \label{eqn:condition2}
A^T\left(\beta(I+B^2)^{-1}-B^{-2}\right)A&&+(\beta-1)\left( I+D^2B^{-2}+D^2 \right)\nonumber \\ &&\hspace{-1in}+DA^T\left(\beta(I+B^2)^{-1}-B^{-2}\right)AD -A^TB^{-2}D-DB^{-2}A \nonumber \\ &&\hspace{-1in}- D\left(A^T\left(\beta (I+B^2)^{-1}-B^{-2}\right)A+(\beta-1)(B^{-2}+I)\right)D \nonumber \\ &&\hspace{-1in}+ A^TB^{-2}D + DB^{-2}A - A^TB^{-2}Z^{-1}B^{-2}A \geq 0.
\end{eqnarray}
Hence, it follows from~(\ref{eqn:condition2}) that (\ref{eqn:condition1.5}) can be simplified as
\begin{eqnarray} \label{eqn:condition2_simple}
A^T\left(\beta(I+B^2)^{-1}-B^{-2}\right)A-A^TB^{-2}Z^{-1}B^{-2}A \geq 0.
\end{eqnarray}
The condition~(\ref{eqn:Z}) is satisfied, for all plants $P\in\p$, if $\beta \geq 1+1/\epsilon^2$, since in this case $\beta (I+B^2)^{-1}-B^{-2}\geq 0$ (recall that any matrix $B$ is diagonal and its diagonal elements are lower-bounded by $\epsilon$). Furthermore, for all $\beta \geq 1+1/\epsilon^2$, it is easy to see that $Z\geq (\beta-1)(B^{-2}+I)$. As a result, it can be shown that the condition~(\ref{eqn:condition2_simple}) is satisfied if
\begin{equation} \label{eqn:condition3}
A^T\left(\beta(I+B^2)^{-1}-B^{-2}-(\beta-1)^{-1} B^{-2}(B^{-2}+I)^{-1}B^{-2} \right)A + (\beta-1) I \geq 0.
\end{equation}
Now, the condition~(\ref{eqn:condition3}) is satisfied if
\begin{eqnarray} \label{eqn:condition3_simple}
\beta(I+B^2)^{-1}-B^{-2}-(\beta-1)^{-1} B^{-2}(B^{-2}+I)^{-1}B^{-2}\geq 0.
\end{eqnarray}
Noting that the matrix $B=\diag(b_{11},\dots,b_{nn})$, one can rewrite~(\ref{eqn:condition3_simple}) as
\begin{eqnarray} \label{eqn:condition3_simple_element}
\frac{\beta}{1+b_{ii}^2}-\frac{1}{b_{ii}^2}-\frac{1}{\beta-1}\frac{1}{b_{ii}^2(1+b_{ii}^2)}\geq 0.
\end{eqnarray}
for all $b_{ii}\geq \epsilon$. Retracing our steps backward, it easy to see that the set
\begin{eqnarray}
\left \{ \beta\;|\; \beta \geq 1+\frac{1}{\epsilon^2} \mbox{ and~(\ref{eqn:condition3_simple_element}) satisfied} \right\}&=& \left \{ \beta \geq \frac{2\epsilon^2+1+\sqrt{4\epsilon^2+1}}{2\epsilon^2} \right\}\nonumber \subseteq \mathcal{M}\nonumber.
\end{eqnarray}
Therefore, we get
\begin{equation} \label{eqn_proof_1}
r_\p(\Gamma^\Delta)=\sup_{P\in\p} \frac{J_P(\Gamma^\Delta(A,B,D))}{J_P(K^*(P))} \leq \frac{2\epsilon^2+1+\sqrt{4\epsilon^2+1}}{2\epsilon^2}.
\end{equation}
\par Now, we have to show that this upper bound can be achieved by a family of plants. Consider a one-parameter family of matrices $\{A(r)\}$ defined as $A(r)=re_je_i^T$ for each $r\in \mathbb{R}$. It is always possible to find indices $i$ and $j$ such that $i\neq j$ and $(s_\p)_{ji}\neq 0$, because of the assumption that there be no isolated node in the plant graph. Let $B=\epsilon I$ and $D=I$. For each $r\in \mathbb{R}$, the matrix $A(r)$ is a nilpotent matrix of degree two, that is, $A(r)^2=0$. Thus, using Lemma~\ref{prop:2}, we get
$$
J_P(K^*_C(P)) = J_P(K^*(P))
$$
for this special plant. The solution to the discrete algebraic Riccati equation in~(\ref{eqn_Riccati}) is
$$
X=\matrix{cc}{A(r)^TA(r)+I & \epsilon A(r)^T \\ \epsilon A(r) & \epsilon^2/(1+\epsilon^2) A(r)^TA(r) + (\epsilon^2+1)I}.
$$
Thus, if we assume that
\begin{equation} \label{eqn_x_0}
x_0=\frac{(\epsilon^2 + 1)(\sqrt{4\epsilon^2 + 1}+ 1)}{2\epsilon r} e_i,
\end{equation}
and
\begin{equation} \label{eqn_w_0}
w_0=\frac{(\epsilon^2 + 1)(\sqrt{4\epsilon^2 + 1} + 1)}{2\epsilon^2r}e_i-e_j,
\end{equation}
the cost of the optimal control design strategy is
\begin{eqnarray} \label{eqn:KP*}
J_P(K^*(P))&=& \frac{(\epsilon^2+1)\sqrt{4\epsilon^2 + 1} + 5\epsilon^2 + 4\epsilon^4 +1}{2\epsilon^2}  \\ &&\hspace{1.2in}+ \frac{(2\epsilon^2 + \sqrt{4\epsilon^2 + 1} + 1)\sqrt{4\epsilon^2 + 1}}{2\epsilon^2r^2}, \nonumber
\end{eqnarray}
and the cost of the deadbeat control design strategy is
\begin{equation}
\begin{split}
J_P(\Gamma^\Delta(A,B,D))&=\frac{(\epsilon^2 + 1)(3\epsilon^2\sqrt{4\epsilon^2 + 1} + 5\epsilon^2 + 4\epsilon^4 + \sqrt{4\epsilon^2 + 1} + 1)}{2\epsilon^4} \label{eqn:Delta} \\ &\hspace{0.0in}+\frac{(\epsilon^2 + 1)(\epsilon^2\sqrt{4\epsilon^2 + 1} + \epsilon^4\sqrt{4\epsilon^2 + 1} + \epsilon^2 + 3\epsilon^4 + 2\epsilon^6)}{2\epsilon^4r^2} .
\end{split}
\end{equation}
This results in
\begin{equation} \label{eqn_proof_2}
\lim_{r\rightarrow \infty} \frac{J_P(\Gamma^\Delta(A,B,D))}{J_P(K^*(P))}=\frac{2\epsilon^2+1+\sqrt{4\epsilon^2+1}}{2\epsilon^2}.
\end{equation}
Equation~(\ref{eqn_proof_1}) together with~(\ref{eqn_proof_2}) conclude the proof.
\end{proof}

\begin{remark} Consider the limited model information design problem given by the plant graph $G_\p$ in Figure~\ref{figure0a}($a$) and the control graph $G_\K$ in Figure~\ref{figure0a}($b$). Theorem~\ref{tho:1} shows that, if we apply the deadbeat control design strategy to this particular problem, the performance of the deadbeat control design strategy, at most, can be $(2\epsilon^2+1+\sqrt{4\epsilon^2+1})/(2\epsilon^2)$ times the cost of the optimal control design strategy~$K^*$.  In~fact, Theorem~\ref{tho:1} states that this relationship between the performance of the deadbeat control design and the optimal control design with full model information holds for a rather general class of systems. For the case that $\mathcal{B}=\{I\}$, the relationship is given by $(3+\sqrt{5})/2\approx 2.62$, so the deadbeat control design strategy is never worse than two or three times the optimal. \end{remark}

With this characterization of $\Gamma^{\Delta}$ in hand, we are now ready to tackle problem~(\ref{eqn:0}).

\section{Plant Graph Influence on Achievable Performance} \label{sec:Gp} In this section, we study the relationship between the plant graph and the achievable closed-loop performance in terms of the competitive ratio as a performance metric and the domination as a partial order on the set of limited model information control design strategies. To this end, we first state and prove two lemmas which will simplify further developments.

\begin{lemma} \label{prop:1} Fix real numbers $a\in \mathbb{R}$ and $b\in \mathbb{R}$. For any $x\in \mathbb{R}$, we have $x^2+(a+bx)^2\geq a^2/(1+b^2).$ \end{lemma}

\begin{proof} Consider the function $x\mapsto x^2+(a+bx)^2$. Since this function is both continuously differentiable and strictly convex, we can find its unique minimizer as $\bar{x}=-ab/(1+b^2)$ by setting its derivative to zero. As a result, we get
$$
x^2+(a+bx)^2\geq \bar{x}^2+(a+b\bar{x})^2=a^2/(1+b^2).
$$
This concludes the proof.
\end{proof}

\begin{lemma} \label{lem:2} Let the design graph $G_\comm$ be a totally disconnected graph, and $G_\K\supseteq G_\p$. Furthermore, assume that node $i$ is not a sink in the plant graph $G_{\p}$. Then, the competitive ratio of a control design strategy $\Gamma\in \comm$ is bounded only if $a_{ij}+b_{ii}(d_\Gamma)_{ij}(A,B,D)=0$ for all $j\neq i$ and all matrices $A\in \A(S_\p)$, $B\in \B(\epsilon)$, and $D\in \D$.
\end{lemma}

\begin{proof} The proof is by contrapositive. Let us assume that there exist matrices $\bar{A}\in \A(S_\p)$, $B\in \B(\epsilon)$, $D\in \D $, and indices $i$ and $j$ such that $i\neq j$ and $\bar{a}_{ij}+b_{ii} (d_\Gamma)_{ij}(\bar{A},B,D)\neq 0$. Let $1\leq \ell\leq n$ be an index such that $\ell\neq i$ and $(s_\p)_{\ell i}\neq 0$ (such an index always exists because node $i$ is not a sink in the plant graph $G_\p$). Define matrix $A$ such that $A_i=\bar{A}_i$, $A_\ell=re_i^T$, and $A_t=0$ for all $t\neq i,\ell$. Because the design graph is a totally disconnected graph, we know that $\Gamma_i(\bar{A},B,D)=\Gamma_i(A,B,D)$. Using the structure of the cost function in~(\ref{cost_easy}) and plant dynamics in~(\ref{eqn_1}), the cost of this control design strategy for $w_0=e_j$ and $x_0=0$ is lower-bounded by
\begin{eqnarray}
J_{(A,B,D,0,e_j)}(\Gamma(A,B,D))&\geq& \left(u_\ell(2)+w_\ell(2)\right)^2+x_\ell(3)^2 \nonumber \\ &=& \left(u_\ell(2)+w_\ell(2)\right)^2+\left(rx_i(2)+b_{\ell\ell}[u_\ell(2)+w_\ell(2)]\right)^2.\nonumber
\end{eqnarray}
Based on Lemma~\ref{prop:1} and the fact that $x_i(2)=(a_{ij}+b_{ii}(d_\Gamma)_{ij}(A,B,D))b_{jj}$ (see~Figure~\ref{figure2}), we get
\begin{eqnarray}
J_{(A,B,D,0,e_j)}(\Gamma(A,B,D))&\geq& r^2x_i(2)^2/(1+b_{\ell\ell}^2) \nonumber \\ &=& (a_{ij}+b_{ii}(d_\Gamma)_{ij}(A,B,D))^2b_{jj}^2 r^2/(1+b_{\ell\ell}^2).\nonumber
\end{eqnarray}
On the other hand, the cost of the deadbeat control design strategy is
\begin{eqnarray}
J_{(A,B,D,0,e_j)}(\Gamma^\Delta(A,B,D))&=&e_j^TB^T(A^TB^{-2}A+B^{-2}+I)Be_j\nonumber \\&=&b_{jj}^2+1+a_{ij}^2b_{jj}^2/b_{ii}^2.\nonumber
\end{eqnarray}
Note that the deadbeat control design strategy is applicable here since the control graph $G_{\K}$ is a supergraph of the plant graph $G_\p$. This gives
\begin{eqnarray} \label{eqn_7}
r_p(\Gamma)&=&\sup_{P \in \p } \frac{J_P (\Gamma(A,B,D))}{J_P (K^*(P))}\nonumber \\&=& \sup_{P \in \p } \left[ \frac{J_P (\Gamma(A,B,D))}{J_P (\Gamma^\Delta(A,B,D))} \frac{J_P (\Gamma^\Delta(A,B,D))}{J_P (K^*(P))} \right] \nonumber \\& \geq & \sup_{P \in \p } \frac{J_P (\Gamma(A,B,D))}{J_P (\Gamma^\Delta(A,B,D))} \\&\geq& \frac{(a_{ij}+b_{ii}(d_\Gamma)_{ij}(A,B,D))^2b_{jj}^2/(1+b_{\ell\ell}^2) }{b_{jj}^2+1+a_{ij}^2b_{jj}^2/b_{ii}^2}\lim_{r\rightarrow \infty}r^2 =\infty.\nonumber
\end{eqnarray}
This inequality proves the statement by contrapositive as the competitive ratio is not bounded in this case.
\end{proof}

\jpfig{0.7}{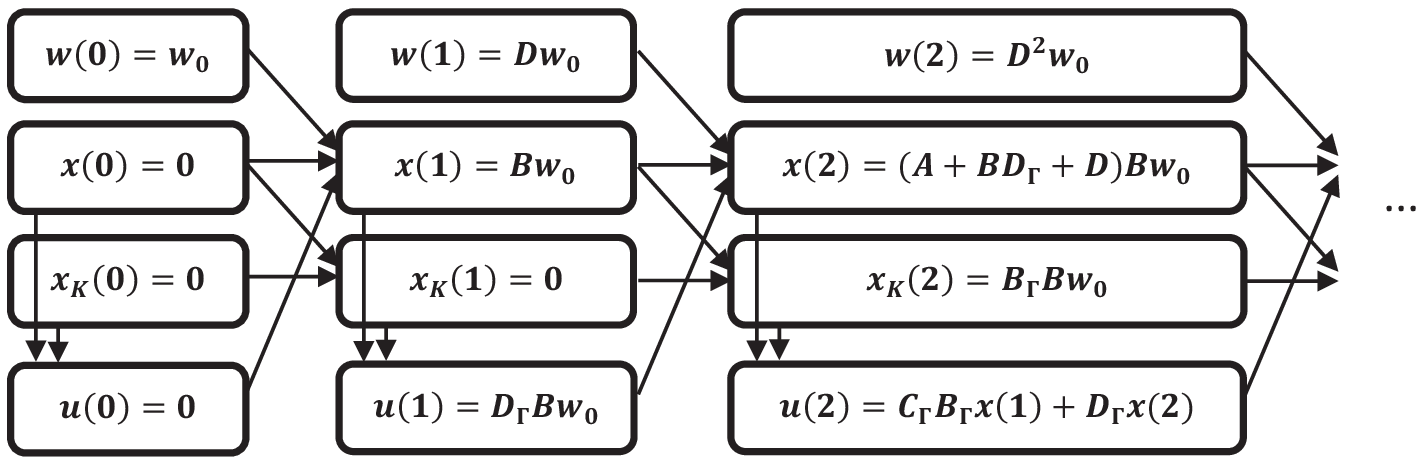}{figure2}{State evolution of the closed-loop system with any control design strategy $\Gamma$ when $x_0=0$.}

\subsection{Plant Graphs without Sinks} First, we assume that there is no sink in the plant graph and try to characterize the optimal control design strategy in terms of the competitive ratio and domination.

\begin{theorem} \label{tho:2}  Let the plant graph $G_{\p}$ contain no sink, the design graph $G_{\mathcal{C}}$ be a totally disconnected graph, and $G_{\K}\supseteq G_\p$. Then, the competitive ratio of any control design strategy $\Gamma\in\comm$ satisfies
$$
r_{\p} (\Gamma)\geq \frac{2\epsilon^2+1+\sqrt{4\epsilon^2+1}}{2\epsilon^2}.
$$
\end{theorem}

\begin{proof} Consider a one-parameter family of matrices $\{A(r)\}$ defined as $A(r)=re_je_i^T$ for each $r\in \mathbb{R}$. It is always possible to find indices $i$ and $j$ such that $i\neq j$ and $(s_\p)_{ji}\neq 0$, because of the assumption that there is no isolated node in the plant graph. Let $B=\epsilon I$ and $D=I$. Let $\Gamma\in\comm$ be a control design strategy with design graph $G_\comm$. Without loss of generality, we can assume that $\gamma_{ji}(A,B,D)=-r/\epsilon$ since otherwise, using Lemma~\ref{lem:2}, we get that $r_\p(\Gamma)$ is infinity, and as a result the inequality in the theorem statement is trivially satisfied. Thus, for each $r\in \mathbb{R}$, the cost of the control design strategy $\Gamma$ for $x_0$ in~(\ref{eqn_x_0}) and $w_0$ in~(\ref{eqn_w_0}) is lower-bounded by
\begin{eqnarray}
J_P(\Gamma(A,B,D)) &\geq& (u_j(0)+w_j(0))^2 + x_j(1)^2 \nonumber \\ &=& \left(\frac{(\epsilon^2 + 1)(\sqrt{4\epsilon^2 + 1}+ 1)}{2\epsilon^2} +1 \right)^2+ \epsilon^2 \nonumber \\ &=&\frac{(\epsilon^2 + 1)(3\epsilon^2\sqrt{4\epsilon^2 + 1} + 5\epsilon^2 + 4\epsilon^4 + \sqrt{4\epsilon^2 + 1} + 1)}{2\epsilon^4}.\nonumber
\end{eqnarray}
On the other hand, for each $r\in \mathbb{R}$, the matrix $A(r)$ is a nilpotent matrix of degree two, that is, $A(r)^2=0$. Consequently, using Lemma~\ref{prop:2}, the cost of the optimal control design strategy $K^*(P)$ for $x_0$ in~(\ref{eqn_x_0}) and $w_0$ in~(\ref{eqn_w_0}) is given by~(\ref{eqn:KP*}). This results in
$$
r_\p(\Gamma)\geq \lim_{r\rightarrow \infty} \frac{J_P(\Gamma(A,B,D))}{J_P(K^*(P))}=\frac{2\epsilon^2+1+\sqrt{4\epsilon^2+1}}{2\epsilon^2}.
$$
\end{proof}

Theorem~\ref{tho:2} shows that the deadbeat control design method $\Gamma^\Delta$ is a minimizer of the competitive ratio $r_{\p}$ as a function over the set of limited model information design methods $\comm$. The following theorem shows that it is also undominated by methods of this type, if and only if, the plant graph $G_{\p}$ has no sink.

\begin{theorem} \label{tho:3} Let the design graph $G_{\mathcal{C}}$ be a totally disconnected graph, and $G_\K\supseteq G_\p$. Then, the control design strategy $\Gamma^\Delta$ is undominated if and only if there is no sink in the plant graph $G_{\p}$. \end{theorem}

\begin{proof} First, we have to prove the sufficiency part of the theorem. Assume that there is no sink in the plant graph. For proving this claim, we are going to prove that for any control design method $\Gamma\in \comm \setminus \{\Gamma^\Delta \}$, there exists a plant $P=(A,B,D,x_0,w_0)\in \p$ such that $J_P(\Gamma(A,B,D))>J_P(\Gamma^\Delta(A,B,D))$. First, assume that there exist matrices $\bar{A}\in \A(S_\p)$, $B\in \B(\epsilon)$, and $D\in \D$ and an index $j$ such that $\bar{A}_j+b_{jj}(D_\Gamma)_j(\bar{A},B,D)+d_{jj}e_j^T\neq 0$. Without loss of generality, we can assume that $\bar{a}_{jj}+b_{jj}(d_\Gamma)_{jj}(\bar{A},B,D)+d_{jj}\neq 0$, because otherwise, using Equation~(\ref{eqn_7}) in the proof of Lemma~\ref{lem:2}, we know that, if there exists $\ell \neq j$ such that $\bar{a}_{j\ell}+b_{jj}(d_\Gamma)_{j\ell}(\bar{A},B,D)\neq 0$, the ratio of the cost of the control design strategy $\Gamma$ to the cost of the deadbeat design strategy $\Gamma^\Delta$ is unbounded. Therefore, the control design strategy $\Gamma$ cannot dominate the deadbeat control design strategy $\Gamma^\Delta$. Pick an index $i\neq j$ such that $(s_\p)_{ij}\neq 0$. It is always possible to pick such index $i$ because there is no sink in the plant graph. Define matrix $A$ such that $A_j=\bar{A}_j$, $A_{i}=re_j^T$, and $A_\ell=0$ for all $\ell\neq i,j$. It should be noted that $\Gamma_{j}(A,B,D)=\Gamma_{j}(\bar{A},B,D)$ because the design graph is a totally disconnected graph. We know that $r+b_{ii}(d_\Gamma)_{ij}(A,B,D)=0$ because otherwise the control design strategy $\Gamma$ cannot dominate the deadbeat control design strategy. The cost of this control design strategy for $w=e_j$ and $x_0=0$ satisfies
\begin{equation}
\begin{split}
J_{P}(\Gamma(A,B,D)) &\geq (u_i(1)+w_i(1))^2+(u_i(2)+w_i(2))^2+x_i(3)^2 \nonumber \\& = r^2b_{jj}^2/b_{ii}^2 + (u_i(2)+w_i(2))^2 + (x_j(2)r+b_{ii}[u_i(2)+w_i(2)])^2, \nonumber
\end{split}
\end{equation}
because of the structure of the cost function~(\ref{cost_easy}) and the plant dynamics~(\ref{eqn_1}). Now, using Lemma~\ref{prop:1}, we have
$$
J_{P}(\Gamma(A,B,D)) \geq r^2b_{jj}^2/b_{ii}^2 + x_j(2)^2 r^2/(1+b_{ii}^2).
$$
As a result
\begin{eqnarray}
&&J_{P}(\Gamma(A,B,D))-J_{P}(\Gamma^\Delta(A,B,D)) \label{eqn:diff} \\&& \hspace{0.3in} \geq (\bar{A}_{jj}+b_{jj}(d_\Gamma)_{jj}(\bar{A},B,D)+d_{jj})^2b_{jj}^2 r^2/(1+b_{ii}^2)-(b_{jj}^2+1+a_{jj}^2),\nonumber
\end{eqnarray}
since $x_j(2)=(\bar{A}_{jj}+b_{jj}(d_\Gamma)_{jj}(\bar{A},B,D)+d_{jj})b_{jj}$ (see~Figure~\ref{figure2}) and
\begin{eqnarray}
J_{(A,B,D,0,e_j)}(\Gamma^\Delta(A,B,D))&=&e_j^TB^T(A^TB^{-2}A+B^{-2}+I)Be_j\nonumber\\&=&b_{jj}^2+1+r^2b_{jj}^2/b_{ii}^2+a_{jj}^2.\nonumber
\end{eqnarray}
Thus, if we pick $r$ large enough, the difference in~(\ref{eqn:diff}) becomes positive, which shows that the control design strategy $\Gamma$ cannot dominate the deadbeat control design strategy $\Gamma^\Delta$. Now, assume that there exist matrices $\bar{A}\in \A(S_\p)$, $B\in \B(\epsilon)$, and $\bar{D}\in \D$ and an index $j$ such that $\bar{A}_j+b_{jj}(D_\Gamma)_j(\bar{A},B,\bar{D})+\bar{d}_{jj}e_j^T= 0$ but $\Gamma_{j}(\bar{A},B,\bar{D})\neq \Gamma_{j}^\Delta(\bar{A},B,\bar{D})$. Define matrix $A$ such that $A_j=\bar{A}_j$ and $A_\ell=0$ for all $\ell\neq j$ and matrix $D$ as $d_{jj}=\bar{d}_{jj}$ and $d_{\ell\ell}=0$ for all $\ell\neq j$. Let $x_0=0$. If there exists an index $i\neq j$ such that $\gamma_{ij}(\bar{A},B,D)\neq \gamma_{ij}^\Delta (\bar{A},B,D)$ pick $w_0=e_i$, otherwise, pick $w_0=e_j$. For this special case, the state of the closed-loop system with the controller $\Gamma(A,B,D)$ is equal to the state of the closed-loop system with the controller $\Gamma^\Delta(A,B,D)$ for the first and the second time-steps (see Figure~\ref{figure2} and Figure~\ref{figure3}). As a result, the state of the subsystem $j$ reaches zero in two time-steps. Now, since $\Gamma_{j}(\bar{A},B,\bar{D})\neq \Gamma_{j}^\Delta(\bar{A},B,\bar{D})$, in the next time-step the state of the subsystem $j$ becomes non-zero again. This results in a performance cost greater than the performance cost of the control design strategy $\Gamma^\Delta$. Thus, the control design $\Gamma^\Delta$ is undominated by the control design method $\Gamma$.

\par Now, we have to prove the necessary part of the theorem. Proving this part is equivalent to proving that if there exists (a sink) $j$ such that for every $i\neq j$, $(s_\p)_{ij}=0$, then there exists a control design strategy $\Gamma$ which can dominate the deadbeat control design strategy. Without loss of generality, let $j=n$; i.e., assume that $(s_\p)_{in}=0$ for all $i\neq n$. In this situation, we can rewrite the matrix $A$ as
$$
A=\matrix{cccc}{ a_{11} & \cdots & a_{1,n-1} & 0 \\  \vdots & \ddots & \vdots & \vdots \\ a_{n-1,1} & \cdots & a_{n-1,n-1} & 0 \\ a_{n1} & \cdots & a_{n,n-1} & a_{nn} },
$$
Define $\bar{x}_0=[x_1(0)\;\cdots\;x_{n-1}(0)]^T$ and $\bar{w}_0=[w_1(0)\;\cdots\;w_{n-1}(0)]^T$. Let $\Gamma(A,B,D)$ be defined as $A_\Gamma(A,B,D)=D$, $C_\Gamma(A,B,D)=I$,
\begin{eqnarray}
B_\Gamma(A,B,D)&=&\matrix{cccc}{-\frac{d_{11}^2}{b_{11}} & \cdots & 0 & 0 \\  \vdots &  \ddots & \vdots & \vdots \\  0 &  \cdots & -\frac{d_{n-1,n-1}^2}{b_{n-1,n-1}} & 0 \\ (b_\Gamma)_{n1} & \cdots & (b_\Gamma)_{n,n-1} & (b_\Gamma)_{nn} }, \nonumber \\
D_\Gamma(A,B,D)&=&\matrix{cccc}{ -\frac{a_{11}+d_{11}}{b_{11}} & \cdots & -\frac{a_{1,n-1}}{b_{11}} & 0 \\ \vdots &  \ddots & \vdots & \vdots \\ -\frac{a_{n-1,1}}{b_{n-1,n-1}} & \cdots & -\frac{a_{n-1,n-1}+d_{n-1,n-1}}{b_{n-1,n-1}} & 0 \\ (d_\Gamma)_{n1} & \cdots & (d_\Gamma)_{n,n-1} & (d_\Gamma)_{nn} }, \nonumber
\end{eqnarray}
where $\bar{B}_\Gamma=[(b_\Gamma)_{n1}\;\cdots\;(b_\Gamma)_{nn}]$ and $\bar{D}_\Gamma=[(d_\Gamma)_{n1}\;\cdots\;(d_\Gamma)_{nn}]$ are tunable gains for the last subsystem. We denote the cost of applying the deadbeat controller to subsystems $1,\dots,n-1$ by $J^{(1)}_{(A,B,D,\bar{x}_0,\bar{w}_0)}$. This cost is independent of the control design parameters $\bar{B}_\Gamma$ and $\bar{D}_\Gamma$, because the last subsystem is a sink and it cannot affect the other subsystems. The overall cost of the controller is
$$
J_{(A,B,x_0,w_0)}(\Gamma(A,B,D))=J^{(1)}_{(A,B,D,\bar{x}_0,\bar{w}_0)}+J^{(2)}_{(A,B,D,x_0,w_0)}(\bar{B}_\Gamma,\bar{D}_\Gamma),
$$
where $J^{(2)}_{(A,B,D,x_0,w_0)}(\bar{B}_\Gamma,\bar{D}_\Gamma)$ is the cost of the controller designed for the last subsystem. This cost $J^{(2)}_{(A,B,D,x_0,w_0)}( \bar{B}_\Gamma , \bar{D}_\Gamma )$ is independent of the rest of the system's model, because the deadbeat (for subsystems $1,\dots,n-1$) cancel out all dependencies in matrix $A$, thus, one can design the optimal controller for the lower part of the system without the model information of the upper part. Now, we can use the method mentioned in Subsection~\ref{subsec_OCCDS} to design the optimal controller for the lower part and find the optimal gains
$$
\bar{B}_\Gamma=\frac{d_{nn}}{b_{nn}}\left((\alpha+1) A_n - D_n\right), \hspace{.2in} \bar{D}_\Gamma=\frac{1}{b_{nn}}\left(\alpha A_n-D_n\right),
$$
where
$$
\alpha=\frac{2}{b_{nn}^2+a_{nn}^2+1+\sqrt{a_{nn}^4+2a_{nn}^2b_{nn}^2-2a_{nn}^2+b_{nn}^4+2b_{nn}^2+1}}-1.
$$
Note that this new control design strategy is always applicable since the control graph $G_\K$ is supergraph of the plant graph $G_\p$. Therefore, there exists a control design strategy which satisfies
$$
J_{(A,B,D,x_0,w_0)}(\Gamma(A,B,D))\leq J_{(A,B,D,x_0,w_0)}(\Gamma^\Delta(A,B,D)),
$$
for all matrices $A\in \A(S_\p)$, $B\in \B(\epsilon)$, and $D\in \D$ and all vectors $x_0\in\mathbb{R}^n$ and $w_0\in\mathbb{R}^n$. Consider the matrix $A\in \A(S_\p)$ such that $A_n=re_n^T$ and $A_\ell=0$ for all $\ell\neq n$. Let $B=\epsilon I$ and $D=I$. For this special system, for all $r>0$, we have
\begin{eqnarray}
J_{(A,B,D,0,e_n)}(\Gamma(A,B,D)) &=&\frac{\sqrt{r^4 + 2r^2\epsilon^2 - 2r^2 + \epsilon^4 + 2\epsilon^2 + 1} + r^2 + \epsilon^2 + 1}{2} \nonumber \\ & <& r^2+\epsilon^2+1 \nonumber \\ &=& J_{(A,B,D,0,e_n)}(\Gamma^\Delta(A,B,D)).\nonumber
\end{eqnarray}
Thus, the control design strategy $\Gamma$ dominates the deadbeat control design strategy $\Gamma^\Delta$.
\end{proof}

\begin{remark} Consider the limited model information design problem given by the plant graph $G'_\p$ in Figure~\ref{figure0a}($a'$), the control graph $G'_\K$ in Figure~\ref{figure0a}($b'$), and the design graph $G'_\comm$ in Figure~\ref{figure0a}($c'$). Theorems~\ref{tho:2} and~\ref{tho:3} show that the deadbeat control design strategy $\Gamma^\Delta$ is the best control design strategy that one can propose based on local model of the subsystems and the plant graph, because the deadbeat control design strategy is the minimizer of the competitive ratio and it is undominated.
\end{remark}

We use the construction in proof of the ``only if'' part of Theorem~\ref{tho:3} to build a control design strategy for the plant graphs with sinks in the next subsection.

\jpfig{0.7}{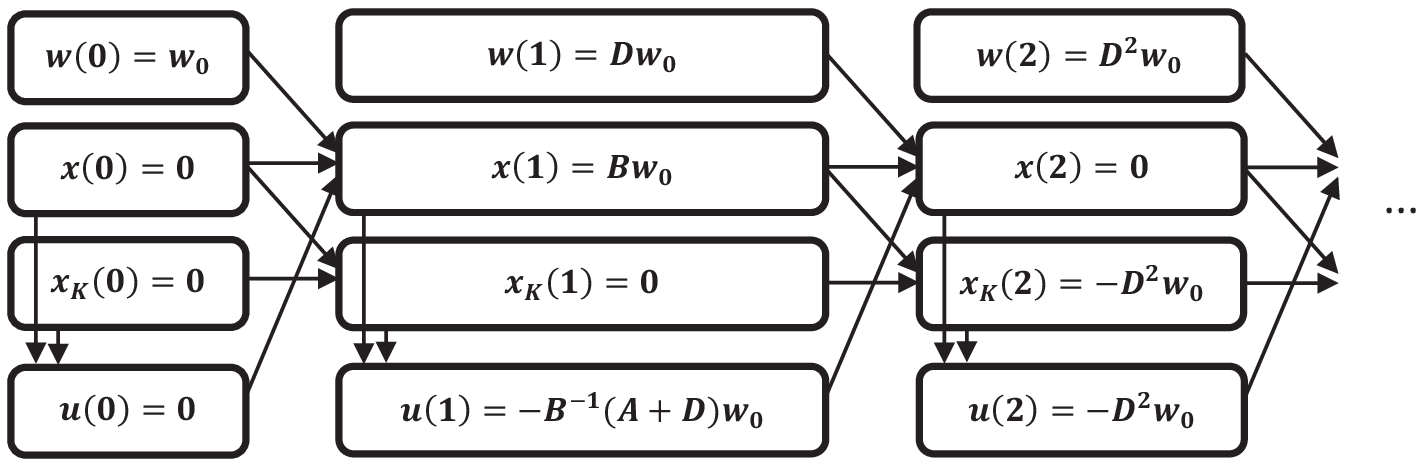}{figure3}{State evolution of the closed-loop system with deadbeat control design strategy $\Gamma^\Delta$ when $x_0=0$.}

\subsection{Plant Graphs with Sinks}
In this section, we study the case where there are $c\geq 1$ sinks in the plant graph. By renumbering the sinks as subsystems number $n-c+1$, $\cdots$, $n$ the matrix $S_\p$ can be written as
\begin{equation} \label{eqn:10}
S_{\p} = \matrix{c|c}{ (S_{\p})_{11} & 0_{(q-c)\times (c)} \\ \hline (S_{\p})_{21} & (S_{\p})_{22} },
\end{equation}
where
$$
(S_\p)_{11}=\matrix{ccc}{ (s_{\p})_{11} & \cdots & (s_{\p})_{1,n-c}\\ \vdots & \ddots & \vdots \\ (s_{\p})_{n-c,1} & \cdots & (s_{\p})_{n-c,n-c} },
$$
$$
(S_\p)_{21}=\matrix{ccc}{ (s_{\p})_{n-c+1,1} & \cdots & (s_{\p})_{n-c+1,n-c}\\ \vdots & \ddots & \vdots \\ (s_{\p})_{n,1} & \cdots & (s_{\p})_{n,n-c} },
$$
and
$$
(S_\p)_{22}=\matrix{ccc}{ (s_{\p})_{n-c+1,n-c+1} & \cdots & 0\\ \vdots & \ddots & \vdots \\ 0 & \cdots & (s_{\p})_{nn} }.
$$
From now on, without loss of generality, we assume that the structure matrix is the one defined in~(\ref{eqn:10}). The control design method $\Gamma^\Theta$ for this type of systems is defined as
\begin{equation} \label{eqn:11}
\Gamma^\Theta(A,B,D)=\left[\begin{array}{c|c} D & B^{-1}D(F(A,B)+I)A-B^{-1}D^2 \\ \hline I  & B^{-1}(F(A,B)A-D) \end{array}\right], \; \forall P\in \p,
\end{equation}
where
$$
F(A,B)=\textrm{diag}(0,\dots,0,f_{n-c+1}(A,B),\dots,f_n(A,B))
$$
and
\begin{equation}
f_i(A,B)=\frac{2}{b_{ii}^2+a_{ii}^2+1+\sqrt{a_{ii}^4+2a_{ii}^2b_{ii}^2-2a_{ii}^2+b_{ii}^4+2b_{ii}^2+1}}-1
\end{equation}
for all $i=n-c+1, \cdots ,n$.

The control design strategy $\Gamma^\Theta$ applies the deadbeat to every subsystem that is not a sink and, for every sink, applies the same optimal control law as if the node was isolated. We will show that when the plant graph contains sinks, the control design method $\Gamma^\Theta$ has, in the worst case, the same competitive ratio as the deadbeat strategy. However, unlike the deadbeat strategy, it has the additional property of being undominated by limited model information methods for plants in $\p$ when the plant graph $G_{\p}$ has sinks.

\begin{theorem} \label{tho:5} Let the plant graph $G_{\p}$ contain at least one sink, and $G_\K\supseteq G_\p$. Then, the competitive ratio of the design method $\Gamma^\Theta$ introduced in~(\ref{eqn:11}) is
$$
r_{\p}(\Gamma^\Theta)=\left\{ \begin{array}{ll} \frac{2\epsilon^2+1+\sqrt{4\epsilon^2+1}}{2\epsilon^2}, & \hspace{0.3in} \mbox{if} \hspace{0.05in} (S_{\p})_{11}\neq 0 \mbox{ is not diagonal}, \\ 1, & \hspace{0.3in} \mbox{if both} \hspace{0.05in} (S_{\p})_{11}=0 \hspace{0.05in} \mbox{and} \hspace{0.05in} (S_{\p})_{22}=0. \end{array} \right.
$$
\end{theorem}

\begin{proof} Based on Theorem~\ref{tho:1}, we know that
\begin{equation} \label{eqn:15}
\hspace{-0.01in} J_{(A,B,D,x_0,w_0)}(K^*(P)) \geq \frac{2\epsilon^2}{2\epsilon^2+1+\sqrt{4\epsilon^2+1}} J_{(A,B,D,x_0,w_0)}(\Gamma^\Delta(A,B,D)),
\end{equation}
and by the proof of the ``only if'' part of Theorem~\ref{tho:3}, we know that
\begin{equation} \label{eqn:16}
J_{(A,B,D,x_0,w_0)}(\Gamma^\Delta(A,B,D)) \geq J_{(A,B,D,x_0,w_0)}(\Gamma^\Theta(A,B,D)),
\end{equation}
for all $x_0\in\mathbb{R}^n$ and $w_0\in\mathbb{R}^n$. Putting~(\ref{eqn:16}) into~(\ref{eqn:15}) results in
$$
J_{(A,B,D,x_0,w_0)}(K^*(P)) \geq \frac{2\epsilon^2}{2\epsilon^2+1+\sqrt{4\epsilon^2+1}} J_{(A,B,D,x_0,w_0)}(\Gamma^\Theta(A,B,D)),
$$
and, therefore, in
$$
\frac{J_{(A,B,D,x_0,w_0)}(\Gamma^\Theta(A,B,D))}{J_{(A,B,D,x_0,w_0)}(K^*(P))} \leq \frac{2\epsilon^2+1+\sqrt{4\epsilon^2+1}}{2\epsilon^2},\; \forall P=(A,B,x_0,w)\in \p.
$$
As a result
$$
r_\p(\Gamma^\Theta)=\sup_{P\in \p} \frac{J_{(A,I,x_0,w)}(\Gamma^\Theta(A,B,D))}{J_{(A,I,x_0,w)}(K_*(P))} \leq \frac{2\epsilon^2+1+\sqrt{4\epsilon^2+1}}{2\epsilon^2}.
$$
If $(S_\p)_{11}$ has an off-diagonal entry, then there exist $1\leq i,j\leq n-c$ and $i\neq j$ such that $(s_\p)_{ij}\neq 0$. Define $A(r)$ such that $A(r)=re_{j}e_{i}^T$. In this case, using the proof of Theorem~\ref{tho:2}, we know
$$
r_\p(\Gamma^\Theta) = \frac{2\epsilon^2+1+\sqrt{4\epsilon^2+1}}{2\epsilon^2},
$$
because the control design $\Gamma^\Theta$ acts as the deadbeat controller on that part of the system. Using both these inequalities proves the statement.
\par If $(S_\p)_{11}=0$ and $(S_\p)_{22}=0$, every matrix $A$ with structure matrix $(S_\p)$ is a nilpotent matrix of degree two. Thus, using Lemma~\ref{prop:2}, we get
$$
J_P(K^*(P))=J_P(K^*_C(P)).
$$
Now, based on the proof of Lemma~\ref{prop:2}, we also know that the optimal controller gain for this plant model is
$$
K_C^*(P)=\left[\begin{array}{c|c} D & D(I+B^2)^{-1}B^{-1}A-B^{-1}D^2 \\ \hline I  & -(I+B^2)^{-1}BA-B^{-1}D \end{array}\right].
$$
For control design strategy $\Gamma^\Theta$, we will have
\begin{eqnarray}
\Gamma^\Theta(A,B,D)&=&\left[\begin{array}{c|c} D & B^{-1}D(B(I+B^2)^{-1}B-I)A-B^{-1}D^2 \\ \hline I  & B^{-1}(B(I+B^2)^{-1}BA-D) \end{array}\right] \nonumber \\ &=&\left[\begin{array}{c|c} D & D(I+B^2)^{-1}B^{-1}A-B^{-1}D^2 \\ \hline I  & -(I+B^2)^{-1}BA-B^{-1}D \end{array}\right] \nonumber
\end{eqnarray}
based on~(\ref{eqn:11}). Thus, $r_\p (\Gamma^\Theta)=1$.
\end{proof}

\begin{theorem} \label{tho:6} Let the plant graph $G_{\p}$ contain at least one sink, the design graph $G_{\mathcal{C}}$ be a totally disconnected graph, and $G_\K\supseteq G_\p$. Then, the competitive ratio of any control design strategy $\Gamma\in\comm$ satisfies
$$
r_{\p}(\Gamma)\geq \frac{2\epsilon^2+1+\sqrt{4\epsilon^2+1}}{2\epsilon^2},
$$
if $(S_{\p})_{11}$ is not diagonal.
\end{theorem}

\begin{proof} First, suppose that $(S_\p)_{11}\neq 0$ and $(S_\p)_{11}$ is not a diagonal matrix, then there exist $1\leq i,j \leq n-c$ and $i\neq j$ such that $(s_\p)_{ij}\neq 0$. Consider the family of matrices $A(r)$ defined by $A(r)=re_{i}e_{j}^T$. Based on Lemma~\ref{lem:2}, if we want to have a bounded competitive ratio, the control design strategy should satisfy $r+b_{ii}(d_\Gamma)_{ij}(A(r),B,D)=0$ (because node $1\leq i\leq n-c$ is not a sink). The rest of the proof is similar to the proof of Theorem~\ref{tho:2}.
\end{proof}

\begin{remark} Combining Theorem~\ref{tho:5} and Theorem~\ref{tho:6} implies that if $(S_{\p})_{11}\neq 0$ is not diagonal (i.e., the nodes that are not sink can affect each other), control design method $\Gamma^\Theta$ is a minimizer of the competitive ratio over the set of limited model information control methods and consequently a solution to the problem (\ref{eqn:0}). Furthermore, if $(S_{\p})_{11}$ and $(S_{\p})_{22}$ are both zero, then the $\Gamma^\Theta$ becomes equal to $K^*$, which shows that, $\Gamma^\Theta$ is a solution to the problem (\ref{eqn:0}), in this case too. The rest of the cases are still open here.  \end{remark}

The next theorem shows that $\Gamma^\Theta$ is a more desirable control design method than the deadbeat when plant graph $G_{\p}$ has sinks, since it is then undominated by limited model information design methods for plants in $\p$.

\begin{theorem} \label{tho:4} Let the plant graph $G_{\p}$ contain at least one sink, the design graph $G_{\mathcal{C}}$ be a totally disconnected graph, and $G_\K \supseteq G_\p$. Then, the control design method $\Gamma^\Theta$ is undominated by all limited model information control design methods. \end{theorem}

\begin{proof} Assume that there are $c\geq 1$ sink in the plant graph. For proving this claim, we are going to prove that for any control design method $\Gamma\in \comm \backslash \{\Gamma^\Theta \}$, there exits a plant $P=(A,B,D,x_0,w_0)\in \p$ such that $J_P(\Gamma(A,B,D))>J_P(\Gamma^\Theta(A,B,D))$. We will proceed in several steps, which require us to partition the set of limited model information control design strategies $\comm$ as follows
$$
\comm= \mathcal{W}_2 \cup \mathcal{W}_1 \cup \mathcal{W}_0 \cup \{\Gamma^\Delta\},
$$
where
\begin{eqnarray}
\mathcal{W}_2:=& \{\Gamma\in \comm \;|\; \exists j, n-c+1 \leq j \leq n, \mbox{ such that }  \Gamma_j(A,B,D)\neq \Gamma^\Theta_j(A,B,D) \},\nonumber
\end{eqnarray}
\begin{eqnarray}
&&\mathcal{W}_1:= \{\Gamma\in  \;\comm \setminus \mathcal{W}_2  \;|\; \exists j, 1\leq j\leq n-c,  \nonumber \\ &&\hspace{1.2in} \mbox{ and } \exists P\in \p,(D_\Gamma)_j(A,B,D)\neq (D_\Gamma^\Theta)_j(A,B,D) \},\nonumber
\end{eqnarray}
and
\begin{eqnarray}
 &&\mathcal{W}_0:= \{\Gamma\in \comm \setminus \mathcal{W}_2 \cup \mathcal{W}_1  \;|\; \exists j, 1\leq j\leq n-c, \exists P\in \p, \nonumber \\&&\hspace{1.2in} \mbox{ such that }  \Gamma_j(A,B,D)\neq \Gamma^\Theta_j(A,B,D) \}.\nonumber
\end{eqnarray}
\par First, we prove that the $\Gamma^\Theta$ is undominated by control design strategies in $\mathcal{W}_2$. We assume that there exist index $n-c+1 \leq j \leq n$ and matrices $\bar{A}\in\A(S_\p)$, $B \in \B(\epsilon)$, $\bar{D}\in \D$ such that $\Gamma_j(\bar{A},B,\bar{D})\neq \Gamma^\Theta_j(\bar{A},B,\bar{D})$. Consider matrices $A$ and $D$ defined as $A_j=\bar{A}_j$ and $A_i=0$ for all $i\neq j$ and $d_{jj}=\bar{d}_{jj}$ and $d_{ii}=0$. For this particular matrix $A$, any $x_0$, and any $w_0$, we know from the proof of the ``only if'' part of Theorem~\ref{tho:3} that $\Gamma^\Theta(A,B,D,x_0,w_0)$ is the globally optimal controller with limited model information. Hence, every other control design method in $\comm$ leads to a controller with greater performance cost than $\Gamma^\Theta$ for this particular type of plants. Therefore, the control design $\Gamma^\Theta$ is undominated by control design methods in $\mathcal{W}_2$.

\par Second, we prove that the control design strategy $\Gamma^\Theta$ is undominated by the control design strategies in $\mathcal{W}_1$. Let $\Gamma$ be a control design strategy in $\mathcal{W}_1$ and let index $1\leq j\leq n-c$ be such that $\bar{A}_j+b_{jj}(D_\Gamma)_j(\bar{A},B,\bar{D})+\bar{d}_{jj}e_j^T\neq0$ for some matrices $\bar{A}\in\A(S_\p)$, $B\in\B(\epsilon)$, and $\bar{D}\in\D$. It is always possible to pick an index $i\neq j$ such that $(s_\p)_{ij}\neq 0$ because node $j$ is not a sink in the plant graph. If $1\leq i\leq n-c$, the proof is the same as the proof of the ``if'' part of Theorem~\ref{tho:3}, therefore, without any loss of generality, we assume that $n-c+1\leq i\leq n$. Again, with the same argument as in the proof of the ``if'' part of Theorem~\ref{tho:3}, without loss of generality, we can assume that $a_{jj}+b_{jj}(d_\Gamma)_{jj}(A,B,D)+d_{jj}\neq 0$ (because otherwise the ratio of the cost the control design strategy $\Gamma$ to the cost of the control design strategy $\Gamma^\Theta$ becomes infinity). Define matrix $A$ such that $A_j=\bar{A}_j$, $A_{i}=re_j^T$, and $A_\ell=0$ for all $\ell\neq i,j$. Let $D\in\D$ be such that $d_{jj}=\bar{d}_{jj}$ and $d_{\ell}=0$ for all $\ell\neq j$. It should be noted that $\Gamma_{j}(A,B,D)= \Gamma_{j} (\bar{A},B,\bar{D})$ because the design graph is a totally disconnected graph. The cost of this control design strategy for $w_0=e_j$ and $x_0=0$ would satisfy
\begin{equation}
\begin{split}
J_{P}(\Gamma(A,B,D)) &\geq (u_i(1)+w_i(1))^2+x_i(2)^2+(u_i(2)+w_i(2))^2+x_i(3)^2 \nonumber\\& = r^2b_{jj}^2/(b_{ii}^2+1) + (u_i(2)+w_i(2))^2 \\ & \hspace{1.4in}+ (x_j(2)r+b_{ii}[u_i(2)+w_i(2)])^2\nonumber \\&\geq (r^2b_{jj}^2+x_j(2)^2r^2)/(1+b_{ii}^2),\nonumber
\end{split}
\end{equation}
This results in
\begin{eqnarray}
&&J_{(A,I,B,D,0,e_j)}(\Gamma(A,B,D))-J_{(A,I,B,D,0,e_j)}(\Gamma^\Theta(A,B,D)) \nonumber\\&&\hspace{.6in} \geq  (a_{jj}+b_{jj}(d_\Gamma)_{jj}(A,B,D)+d_{jj})^2b_{jj}^2r^2/(1+b_{ii}^2)-\kappa(A_j,b_{jj}).\nonumber
\end{eqnarray}
where $\kappa(A_j,b_{jj})$ is only a function $A_j$ and $b_{jj}$ and represents the part of the cost of the control design strategy $\Gamma^\Theta$ that is related to subsystem $j$ only. If we pick $r$ large enough, the difference would become positive, which shows that the control design strategy $\Gamma$ cannot dominate the control design strategy $\Gamma^\Theta$.

\par Finally, we prove that the control design strategy $\Gamma^\Theta$ is undominated by the control design strategies in $\mathcal{W}_0$. The same argument as in the proof of the ``if'' part of Theorem~\ref{tho:3} holds here too.
\end{proof}

\begin{remark} Consider the limited model information design problem given by the plant graph $G_\p$ in Figure~\ref{figure0a}($a$), the control graph $G'_\K$ in Figure~\ref{figure0a}($b'$), and the design graph $G'_\comm$ in Figure~\ref{figure0a}($c'$). Theorems~\ref{tho:5},~\ref{tho:6}, and~\ref{tho:4} together show that, the control design strategy $\Gamma^\Theta$ is the best control design strategy that one can propose based on local subsystems' model and the plant graph, because the control design strategy $\Gamma^\Theta$ is a minimizer of the competitive ratio and it is undominated.
\end{remark}

\section{Design Graph Influence on Achievable Performance}  \label{sec:Gc}  In the previous section, we approached the optimal control design under limited model information when $G_{\mathcal{C}}$ is a totally disconnected graph. The next step is to determine the necessary amount of the model information needed in each subcontroller to be able to setup a control design strategy with a smaller competitive ratio than the deadbeat control design strategy. We tackle this question here.

\begin{theorem} \label{tho:8} Let the plant graph $G_\p$ and the design graph $G_\comm$ be given, and $G_\K\supseteq G_\p$. Assume that the plant graph $G_\p$ contains the path $i \rightarrow j \rightarrow \ell$ with distinct nodes $i$, $j$, and $\ell$ while $(\ell,j)\notin E_\comm$. Then, we have
$$
r_{\p}(\Gamma)\geq \frac{2\epsilon^2+1+\sqrt{4\epsilon^2+1}}{2\epsilon^2}.
$$
\end{theorem}

\begin{proof} Let $i$, $j$, and $k$ be three distinct nodes such that $(s_{\mathcal{P}})_{ji}\neq 0$ and $(s_{\mathcal{P}})_{\ell i}\neq 0$ (i.e., the path $i \rightarrow j \rightarrow \ell$ is contained in the plant graph $G_{\mathcal{P}}$). Define the 2-parameter family of matrices $A(r,s)=re_je_i^T+se_\ell e_j^T$. Let $B=\epsilon I$, $D=I$, and $\Gamma\in \comm$ be a limited model information with design graph $G_{\mathcal{C}}$. The cost of this control design strategy for $w_0=e_i$ and $x_0=0$ satisfies
\begin{eqnarray}
J_{(A,B,D,0,e_j)}(\Gamma(A,B,D))&\geq& \left(u_\ell(2)+w_\ell(2)\right)^2+x_\ell(3)^2 \nonumber \\ &=& \left(u_\ell(2)+w_\ell(2)\right)^2+\left(sx_j(2)+\epsilon[u_\ell(2)+w_\ell(2)]\right)^2,\nonumber
\end{eqnarray}
because of the structure of the cost function in~(\ref{cost_easy}) and the system dynamic in~(\ref{eqn_1}). Now, using Lemma~\ref{prop:1} and the fact that $x_j(2)=(r+\epsilon(d_\Gamma)_{ji}(r))\epsilon$ (see~Figure~\ref{figure2}), we get
\begin{eqnarray}
J_{(A,B,D,0,e_j)}(\Gamma(A,B,D))&\geq& s^2x_j(2)^2/(1+\epsilon^2) \nonumber \\ &=& (r+\epsilon(d_\Gamma)_{ji}(r))^2\epsilon^2 s^2/(1+\epsilon^2).\nonumber
\end{eqnarray}
Note that $(d_\Gamma)_{ji}(r)$ is only a function of $r$ and not $s$ since $(\ell,j)\notin E_\comm$. On the other hand, the cost of the deadbeat control design strategy is
\begin{eqnarray}
J_{(A,B,D,0,e_j)}(\Gamma^\Delta(A,B,D))&=&e_i^TB^T(A^TB^{-2}A+B^{-2}+I)Be_i\nonumber \\&=&\epsilon^2+1+r^2.\nonumber
\end{eqnarray}
Note that the deadbeat control design strategy is applicable here since the control graph $G_{\K}$ is a supergraph of the plant graph $G_\p$. Using~(\ref{eqn_7}), we get
\begin{eqnarray} \label{eqn_71}
r_p(\Gamma) &\geq& \frac{(r+\epsilon(d_\Gamma)_{ji}(r))^2\epsilon^2/(1+\epsilon^2) }{\epsilon^2+1+r^2}\lim_{s\rightarrow \infty}s^2.\nonumber
\end{eqnarray}
Using~(\ref{eqn_71}) it is easy to see that the competitive ratio $r_\p(\Gamma)$ is bounded only if $r+\epsilon(d_\Gamma)_{ji}(r)=0$, for all $r\in\mathbb{R}$. Therefore, there is no loss of generality in assuming that $(d_\Gamma)_{ji}(r)=-r/\epsilon$ because otherwise the $r_\p(\Gamma)$ is infinity and the inequality in the statement of the theorem is trivially satisfied. Now, let us fix $s=0$ and use the notation $A(r)=re_je_i^T$. Since the parameters of the subsystem $j$ is not changed and $(\ell,j)\notin E_\comm$, we have $(d_\Gamma)_{ji}(r)=-r/\epsilon$. Therefore, for each $r\in \mathbb{R}$, similar to the proof of Theorem~\ref{tho:2}, the cost of the control design strategy $\Gamma$ for $x_0$ in~(\ref{eqn_x_0}) and $w_0$ in~(\ref{eqn_w_0}) is lower-bounded by
\begin{eqnarray}
J_P(\Gamma(A,B,D)) \geq \frac{(\epsilon^2 + 1)(3\epsilon^2\sqrt{4\epsilon^2 + 1} + 5\epsilon^2 + 4\epsilon^4 + \sqrt{4\epsilon^2 + 1} + 1)}{2\epsilon^4},\nonumber
\end{eqnarray}
On the other hand, for each $r\in \mathbb{R}$, the matrix $A(r)$ is a nilpotent matrix of degree two, that is, $A(r)^2=0$. Similar to the proof of Theorem~\ref{tho:2}, for $x_0$ in~(\ref{eqn_x_0}) and $w_0$ in~(\ref{eqn_w_0}), we get
\begin{small}
\begin{equation}
J_P(K^*(P))= \frac{(\epsilon^2+1)\sqrt{4\epsilon^2 + 1} + 5\epsilon^2 + 4\epsilon^4 + 1}{2\epsilon^2}+ \frac{(2\epsilon^2 + \sqrt{4\epsilon^2 + 1} + 1)\sqrt{4\epsilon^2 + 1}}{2\epsilon^2r^2}, \nonumber
\end{equation}
\end{small}
since $J_P(K^*(P))=J_P(K^*_C(P))$ according to Lemma~\ref{prop:2}. This results in
$$
r_\p(\Gamma)\geq \lim_{r\rightarrow \infty} \frac{J_P(\Gamma(A,B,D))}{J_P(K^*(P))}=\frac{2\epsilon^2+1+\sqrt{4\epsilon^2+1}}{2\epsilon^2}. \nonumber
$$
This finishes the proof. \end{proof}

\begin{remark} Consider the limited model information design problem given by the plant graph $G'_\p$ in Figure~\ref{figure0a}($a'$), the control graph $G_\K$ in Figure~\ref{figure0a}($b$), and the design graph $G_\comm$ in Figure~\ref{figure0a}($c$). Theorem~\ref{tho:8} shows that, because the plant graph $G_\mathcal{P}$ contains the path $2\rightarrow1\rightarrow4$ but the design graph $G_\mathcal{C}$ does not contain $4\rightarrow 1$, the competitive ratio of any control design strategy $\Gamma\in\comm$ would be greater than or equal to $r_\p(\Gamma^\Delta)$.
\end{remark}

\begin{remark} Theorem~\ref{tho:8} shows that, when $G_{\mathcal{P}}$ and $G_\K$ is a complete graph, achieving a better competitive ratio than the deadbeat design strategy requires each subsystem to have full knowledge of the plant model when constructing each subcontroller. \end{remark}

\section{Proportional-Integral Deadbeat Control Design Strategy} \label{sec_3.5}
In this section, we use some of the results of the paper on familiar control design problems like constant-disturbance rejection and step reference-tracking.

\subsection{Constant-Disturbance Rejection}
For the case of constant-disturbance rejection, we can model the disturbance as in~(\ref{eqn_2}) with matrix $D=I$. For each plant $P=(A,B,I,x_0,w_0)\in \p$, the deadbeat controller design strategy is
$$
\Gamma^\Delta(A,B,I)\triangleq\left[\begin{array}{c|c} I & -B^{-1} \\ \hline I & -B^{-1}(A+I) \end{array}\right],
$$
This controller can be realized as
$$
u(k)=-B^{-1}Ax(k)-B^{-1}\sum_{i=0}^{k} x(i).
$$
which is a proportional-integral controller. Thus, we call the restricted mapping $\Gamma^\Delta_{\mbox{\scriptsize{const}}}:\A(S_\p)\times \B(\epsilon) \rightarrow \K(S_\K)$, defined as $\Gamma^\Delta_{\mbox{\scriptsize{const}}}(A,B)=\Gamma^\Delta(A,B,I)$, the proportional-integral deadbeat control design strategy. The proportional term regulates the states of the system and the integral term compensates for the disturbance. For instance, in this case, Theorem~\ref{tho:2} shows that when the plant graph $G_\p$ contains no sink and the design graph $G_\comm$ is a totally disconnected graph, the deadbeat proportional-integral control design strategy is an undominated minimizer of the competitive ratio. Note that the integral part of this control design strategy is fully decentralized and the proportional part only needs the neighboring subsystems state-measurements.

\subsection{Step Reference-Tracking}
Consider the case that we are interested in tracking a constant reference signal $r\in\mathbb{R}^{n}$. We need to define the difference $\bar{x}(k)=x(k)-r$ which gives
$$
\bar{x}(k+1)=x(k+1)-r=Ax(k)+Bu(k)-r= A \bar{x}(k)+Bu(k)+Ar-r.
$$
Now if the subsystems do not want to share the reference points with each other, we can think of the additional term $Ar-r$ as the constant-disturbance vector $w(k)=B^{-1}(Ar-r)$. Thus, we have
$$
\bar{x}(k+1)=A \bar{x}(k)+B(u(k)+w(k)).
$$
The subsystems only need to transmit the relative error between the state-measurements and reference points. In this case, we can use the cost function
\begin{equation} \label{eqn:cost:ref}
J_P (K)=\sum_{k=0}^\infty [\bar{x}(k)^T\bar{x}(k)+(u(k)+w(k))^T(u(k)+w(k))],
\end{equation}
to make sure that the error $\bar{x}(k)$ goes to zero as time tends to infinity. Note that if we want to have a complete state regulation $\lim_{k\rightarrow \infty} \bar{x}(k)=0$, the control signal should have a limit as
$$
\lim_{k\rightarrow \infty} u(k)=-B^{-1}(Ar-r).
$$
Thus, the second term of the cost function~(\ref{eqn:cost:ref}) only penalizes the difference of the control signal and its steady-state value.

\section{Conclusions} \label{sec_4}
We studied the design of optimal disturbance-rejection and servomechanism dynamic controllers under limited plant model information. We investigated the relationship between closed-loop performance and the control design strategies with limited model information using the performance metric called competitive ratio. We found an explicit minimizer of the competitive ratio and showed that this minimizer is also undominated. This optimal control design is a dynamic control design strategy composed of a static part for regulating the state of the system and a dynamic part for canceling the effect of the disturbance. Possible future work will focus on extending the present framework to situations where the subsystems and disturbances are not scalar.

\bibliography{ref}
\bibliographystyle{ieeetr}

\end{document}